\documentclass[10pt]{amsart}
\usepackage{amssymb}
\usepackage{bm}
\usepackage{graphicx}
\usepackage[centertags]{amsmath}
\usepackage{amsfonts}
\usepackage{amsthm}
\usepackage{amsbsy}
\usepackage{mathtools}
\usepackage{mathrsfs}
\usepackage{cases}
\usepackage{xfrac}
\usepackage[all]{xy}
\usepackage{hyperref}
\newtheorem{thm}{Theorem}[section]

\newtheorem{lem}[thm]{Lemma}

\newtheorem{prop}[thm]{Proposition}
\newtheorem{claim}[thm]{Claim}
\newtheorem{fact}[thm]{Fact}

\newtheorem{defn}[thm]{Definition}
\theoremstyle{definition}

\newcommand{\rr}{\mathbb{R}}
\newcommand{\nn}{\mathbb{N}}
\newcommand{\ee}{\varepsilon}

\newcommand{\meg}{\geqslant}
\newcommand{\mik}{\leqslant}
\newcommand{\ave}{\mathbb{E}}

\newcommand{\cala}{\mathcal{A}}
\newcommand{\calb}{\mathcal{B}}

\newcommand{\calf}{\mathcal{F}}
\newcommand{\calg}{\mathcal{G}}
\newcommand{\calh}{\mathcal{H}}
\newcommand{\cals}{\mathcal{S}}
\newcommand{\calp}{\mathcal{P}}
\newcommand{\calq}{\mathcal{Q}}
\newcommand{\calr}{\mathcal{R}}

\newcommand{\bbx}{\boldsymbol{X}}
\newcommand{\bmu}{\boldsymbol{\mu}}
\newcommand{\bbs}{\boldsymbol{\Sigma}}
\newcommand{\bnu}{\boldsymbol{\nu}}
\newcommand{\bpsi}{\boldsymbol{\psi}}
\newcommand{\bfca}{\mathbf{A}}

\newcommand{\bx}{\mathbf{x}}

\newcommand{\bfu}{\mathbf{f}}

\begin{document}

\title{$L_p$ regular sparse hypergraphs}

\author{Pandelis Dodos, Vassilis Kanellopoulos and Thodoris Karageorgos}

\address{Department of Mathematics, University of Athens, Panepistimiopolis 157 84, Athens, Greece}
\email{pdodos@math.uoa.gr}

\address{National Technical University of Athens, Faculty of Applied Sciences,
Department of Mathematics, Zografou Campus, 157 80, Athens, Greece}
\email{bkanel@math.ntua.gr}

\address{Department of Mathematics, University of Athens, Panepistimiopolis 157 84, Athens, Greece}
\email{tkarageo@math.uoa.gr}

\thanks{2010 \textit{Mathematics Subject Classification}: 05C35, 05C65,  46B25, 60G42.}
\thanks{\textit{Key words}: sparse hypergraphs, removal lemma, pseudorandomness.}


\begin{abstract}
We study sparse hypergraphs which satisfy a mild pseudorandomness condition known as \textit{$L_p$ regularity}.
We prove appropriate regularity and counting lemmas, and we extend the relative removal lemma of Tao in this setting.
This answers a question from \cite{BCCZ}.
\end{abstract}

\maketitle

\tableofcontents


\section{Introduction}

\numberwithin{equation}{section}

\noindent 1.1. In this paper we study a family of sparse weighted uniform hypergraphs which satisfy a mild pseudorandomness condition.
Our primary objective is to prove appropriate regularity and counting lemmas, and to extend the relative removal lemma for this class
of hypergraphs. Beside their intrinsic interest, these types of results have found significant applications in number theory and
Ramsey theory. To a large extend, the present paper was motivated by these applications.
\medskip

\noindent 1.2. To put our discussion in a proper context it is useful to recall one of the most well-known
pseudorandomness conditions for graphs, introduced in \cite{Ko,KoR}. Specifically, let $G=(V,E)$ be a finite graph
and let $p\coloneqq |E|/{|V| \choose 2}$ denote the edge density of $G$; the reader should have in mind that we are
interested in the case where $G$ is sparse, that is, in the regime $p=o(|V|^2)$. Also let $D\meg 1$ and $0<\gamma\mik 1$,
and recall that the graph $G$ is said to be \emph{$(D,\gamma)$-bounded} provided that for every pair $X,Y$ of disjoint
subsets of $V$ with $|X|\meg \gamma |V|$ and $|Y|\meg \gamma |V|$, we have $|E\cap (X\times Y)|\mik Dp |X|\, |Y|$. This
natural condition expresses the fact that the graph $G$ has ``no large dense spots"\!, and is satisfied by several
models of sparse random graphs (see, e.g., \cite{BR}).
\medskip

\noindent 1.3. The notion of pseudorandomness which we will consider in this paper is an extension of the aforementioned boundedness
hypothesis. It is relevant also in the continuous setting---e.g., in the context of graph limits \cite{L}---and as such we will be
working  with the following measure theoretic structures. As usual, for every positive integer $n$ we set $[n]\coloneqq\{1,\dots,n\}$.
\begin{defn}[\cite{Tao1}] \label{d1.1}
A \emph{hypergraph system} is a triple
\begin{equation} \label{e1.1}
\mathscr{H}=(n,\langle(X_i,\Sigma_i,\mu_i):i\in [n]\rangle,\calh)
\end{equation}
where $n$ is a positive integer, $\langle(X_i,\Sigma_i,\mu_i): i\in [n]\rangle$ is a finite sequence of probability spaces
and $\calh$ is a hypergraph on $[n]$. If\, $\calh$ is $r$-uniform, then $\mathscr{H}$ will be called
an \emph{$r$-uniform hypergraph system}. On the other hand, if $\mu_i(A)\mik \eta$ for every $i\in [n]$ and every
atom\footnote{Recall that an \emph{atom} of a probability space is a measurable event with positive probability which contains
no measurable event with smaller positive probability.} $A$ of $(X_i,\Sigma_i,\mu_i)$, then $\mathscr{H}$ will be called
\emph{$\eta$-nonatomic}.
\end{defn}
Observe that a hypergraph system which consists of discrete probability spaces is just a vertex-weighted $n$-partite hypergraph.
Also note that if $\mu$ is the uniform probability measure on a nonempty finite set $X$, then every atom of $(X,\mu)$
has probability $|X|^{-1}$. Therefore, $\eta$-nonatomicity is just an asymptotic condition.

At this point we need to recall some basic definitions. Given a hypergraph system
$\mathscr{H}=(n,\langle(X_i,\Sigma_i,\mu_i):i\in [n]\rangle,\calh)$ by $(\bbx,\bbs,\bmu)$ we denote the product of the spaces
$\langle (X_i,\Sigma_i,\mu_i): i\in [n]\rangle$. More generally, for every nonempty $e\subseteq [n]$ let $(\bbx_e,\bbs_e,\bmu_e)$
be the product of the spaces $\langle (X_i,\Sigma_i,\mu_i): i\in e\rangle$ and observe that the $\sigma$-algebra $\bbs_e$ can be
``lifted" to $\bbx$ by setting
\begin{equation} \label{e1.2}
\calb_e=\big\{ \pi^{-1}_e(\mathbf{A}): \mathbf{A}\in\bbs_e\big\}
\end{equation}
where $\pi_e\colon \bbx\to\bbx_e$ is the natural projection. If, in addition, $|e|\meg 2$, then let
$\partial e=\{e'\subseteq e: |e'|=|e|-1\}$ and define
\begin{equation} \label{e1.3}
\cals_{\partial e}\coloneqq \Big\{ \bigcap_{e'\in\partial e} A_{e'}: A_{e'}\in\calb_{e'} \text{ for every } e'\in\partial e\Big\}
\subseteq \calb_e.
\end{equation}
The family $\cals_{\partial e}$ is very easy to grasp if $n=2$ and $e=\{1,2\}$. It consists of all measurable
rectangles of $X_1\times X_2$, that is,
\[ \cals_{\partial e}=\{A\times B: A\in\Sigma_1 \text{ and } B\in\Sigma_2\}.\]
On the other hand, if $n=3$ and $e=\{1,2,3\}$, then observe that
\[ \cals_{\partial e} = \{ A\cap B\cap C: A\in\calb_{\{1,2\}}, B\in\calb_{\{2,3\}} \text{ and } C\in\calb_{\{1,3\}}\} \]
where $\calb_{\{1,2\}}=\{\mathbf{A}\times X_3: \mathbf{A}\subseteq X_1\times X_2 \text{ is measurable}\}$,
and similarly for $\calb_{\{2,3\}}$ and $\calb_{\{1,3\}}$. In general, notice that every member of $\cals_{\partial e}$
is the intersection of events which depend on fewer coordinates, and so it is useful to view the elements of $\cals_{\partial e}$
as ``lower-complexity" events. Much of the interest in these families of sets stems from the fact that they are ``characteristic"
for a number of structural properties of hypergraphs (see, e.g., \cite{DK2,TV}).
\medskip

\noindent 1.4. We are now in a position to introduce the class of weighted hypergraphs which we will consider in this paper.
It was recently defined\footnote{Actually, in \cite{BCCZ} only $L_p$ regular graphs were considered, but it
is straightforward to extend the definition to hypergraphs.} in \cite{BCCZ}.
\begin{defn} \label{d1.2}
Let $\mathscr{H}=(n,\langle(X_i,\Sigma_i,\mu_i):i\in [n]\rangle,\calh)$ be a hypergraph system. Also let $C,\eta>0$ and
$1\mik p\mik \infty$, and let $e\in \calh$ with $|e|\meg 2$. A random variable $f\in L_1(\bbx,\calb_e,\bmu)$ is said to be
\emph{$(C,\eta,p)$-regular} $($or simply \emph{$L_p$ regular} if\, $C$ and $\eta$ are understood$)$ provided that for every
partition $\calp$ of\, $\bbx$ with $\calp\subseteq \cals_{\partial e}$ and $\bmu(A)\meg \eta$ for every $A\in\calp$ we have
\begin{equation} \label{e1.4}
\|\ave(f\, | \, \cala_{\calp})\|_{L_p} \mik C
\end{equation}
where $\cala_{\calp}$ is the $\sigma$-algebra on\, $\bbx$ generated by $\calp$ and\, $\ave(f\, | \, \cala_{\calp})$ is the
conditional expectation of $f$ with respect to $\cala_{\calp}$.
\end{defn}
The main point in Definition \ref{d1.2} is that, even though we make no assumption on the existence of moments, an $L_p$ regular
random variable \emph{behaves} like a function in $L_p$ as long as we project it on sufficiently ``nice" $\sigma$-algebras of $\bbx$.

Notice that $L_p$ regularity becomes weaker as $p$ becomes smaller. In particular, the case ``$p=1$" is essentially of no interest
since every integrable random variable is $L_1$ regular. On the other hand, in the context of graphs, $L_{\infty}$ regularity reduces
to the boundedness hypothesis that we mentioned above. Indeed, it is not hard to see that a bipartite graph $G=(V_1,V_2,E)$ with
edge density $p$ is $(D,\gamma)\text{-bounded}$ for some $D,\gamma$ if and only if the random variable $\mathbf{1}_E/p$ is $L_{\infty}$
regular. (Here, we view $V_1$ and $V_2$ as discrete probability spaces equipped with the uniform probability measures.) For weighted
hypergraphs, however, $L_{\infty}$ regularity is a more subtle property. It is implied by the pseudorandomness conditions appearing
in the work of Green and Tao \cite{GT1,GT2}, though closer to the spirit of this paper is the work of Tao in \cite{Tao2}. We will
comment further on these connections in Subsection 2.2.

Between the above extremes there is a large class of sparse weighted hypergraphs (namely, those which are $L_p$ regular for some
$1<p<\infty$) which are, as we shall see, particularly well-behaved.


\section{Statement of the main results}

\numberwithin{equation}{section}

\noindent 2.1. \textbf{Regularity lemma.} We first recall some pieces of notation and terminology.
We say that a function $F\colon \nn\to\rr$ is a \emph{growth function} provided that: (i) $F$ is increasing, and (ii) $F(n)\meg n+1$
for every $n\in\nn$. Moreover, for every hypergraph system $\mathscr{H}=(n,\langle(X_i,\Sigma_i,\mu_i):i\in [n]\rangle,\calh)$,
every $e\in \calh$ with $|e|\meg 2$ and every $f\in L_1(\bbx,\calb_e,\bmu)$ let
\begin{equation} \label{e2.1}
\|f\|_{\cals_{\partial e}} =\sup\Big\{ \big| \int_A f\, d\bmu \big|: A\in\cals_{\partial e}\Big\}
\end{equation}
where $\cals_{\partial e}$ is as in \eqref{e1.3}. The quantity $\|f\|_{\cals_{\partial e}}$ is a measure of ``uniformity" of $f$
and it has been extensively used in extremal combinatorics (see, e.g., \cite{FK,Go1,L,Tao1}). It is sometimes referred to in the
literature as the \emph{cut norm} of $f$.

The following decomposition of $L_p$ regular random variables is the first main result of this paper.
As in Definition \ref{d1.2}, for every partition $\calp$ of a nonempty set $X$ by $\cala_{\calp}$
we denote the $\sigma$-algebra on $X$ generated by $\calp$.
\begin{thm} \label{t2.1}
Let $n,r\in\nn$ with $n\meg r\meg 2$, and let $C>0$ and $1<p\mik \infty$. Also let $F\colon \nn\to\rr$ be a growth function and $0<\sigma\mik 1$.
Then there exists a positive integer\, $\mathrm{Reg}=\mathrm{Reg}(n,r,C,p,F,\sigma)$ such that, setting\, $\eta=1/\mathrm{Reg}$, the following holds.
Let $\mathscr{H}=(n,\langle(X_i,\Sigma_i,\mu_i):i\in [n]\rangle,\calh)$ be an $\eta$-nonatomic, $r$-uniform hypergraph system.
For every $e\in\calh$ let $f_e\in L_1(\bbx,\calb_e,\bmu)$ be nonnegative and $(C,\eta,p)$-regular. Then there exist
\begin{enumerate}
\item[(a)] a positive integer $M$ with $M\mik \mathrm{Reg}$,
\item[(b)] for every $e\in\calh$ a partition $\calp_{\!e}$ of\, $\bbx$ with $\calp_{\!e}\subseteq \cals_{\partial e}$ and $\bmu(A)\meg 1/M$
for every $A\in\calp_{\!e}$, and
\item[(c)] for every $e\in\calh$ a refinement $\calq_e$ of\, $\calp_{\!e}$ with $\calq_e\subseteq \cals_{\partial e}$ and $\bmu(A)\meg \eta$
for every $A\in\calq_e$,
\end{enumerate}
such that for every $e\in\calh$, writing $f_e=f^e_{\mathrm{str}}+ f^e_{\mathrm{err}}+f^e_{\mathrm{unf}}$ with
\begin{equation} \label{e2.2}
f^e_{\mathrm{str}}=\ave(f_e\,|\,\cala_{\calp_{\!e}}), \
f^e_{\mathrm{err}}=\ave(f_e\,|\,\cala_{\calq_e})-\ave(f_e\,|\,\cala_{\calp_{\!e}}), \
f^e_{\mathrm{unf}}=f_e-\ave(f_e\,|\,\cala_{\calq_e}),
\end{equation}
we have the estimates
\begin{equation} \label{e2.3}
\|f^e_{\mathrm{str}}\|_{L_p} \mik C, \ \
\|f^e_{\mathrm{err}}\|_{L_{p^{\dagger}}}\mik \sigma \ \text{ and } \
\|f^e_{\mathrm{unf}}\|_{\cals_{\partial e}}\mik \frac{1}{F(M)}
\end{equation}
where $p^{\dagger}=\min\{2,p\}$.
\end{thm}
Note that, unless $p=\infty$, the structured part of the above decomposition (namely, the function $f^e_{\mathrm{str}}$)
is not uniformly bounded. This is an intrinsic feature of $L_p$ regular hypergraphs, and it is an important difference
between $\text{Theorem \ref{t2.1}}$ and several related results (see, e.g., \cite{BR,CCF,CFZ,Go2,GT1,Ko,RTTV,Tao3,TZ1}).
Observe, however, that, by part (b) and \eqref{e2.3}, one has a very good control on the correlation between $f^e_{\mathrm{str}}$
and $f^{e'}_{\mathrm{unf}}$ for every $e,e'\in\calh$. Hence, by appropriately selecting the growth function $F$,
we can force the function $f^e_{\mathrm{str}}$ to behave like a bounded function for many practical purposes.

The proof of Theorem \ref{t2.1} is given in Section 5 and relies in an essential way on a uniform convexity estimate for martingale
difference sequences in $L_p$ spaces (Proposition \ref{p3.1} in the main text), and on a H\"{o}lder-type inequality for $L_p$ regular
random variables which is presented in Section 4. The argument is technically demanding mainly because, on the one hand, we aim to
the strong norm estimates in \eqref{e2.3} and, on the other hand, we need to ensure that the cells of each partition $\calp_{\!e}$
are fairly large. The reader is also referred to \cite{DKK1} where a similar method is used in a simpler (but closely related) context,
as well as to \cite{BK} where the algorithmic aspects of the proof of Theorem \ref{t2.1} are analysed and applications are given for
sparse instances of NP-hard problems.
\medskip

\noindent 2.2. \textbf{Relative counting and removal lemmas.} The second part of this paper deals with the existence
of counting and removal lemmas for $L_p$ regular hypergraphs. Firstly, it is necessary to note that $L_p$ regularity
alone is not strong enough to yield results of this type, even if we restrict our attention to $L_{\infty}$ regular
graphs. It is in fact a rather delicate question to find natural conditions which imply sparse versions of the counting
and removal lemmas\footnote{Recently, there was progress in this direction for sparse random graphs; see \cite{ST} and
the references therein.}. In this direction, a powerful transference method was introduced by Green and Tao \cite{GT1,GT2}
which can be applied to relatively dense subsets of several sufficiently pseudorandom combinatorial structures.
For hypergraph systems, the method was developed by Tao in \cite{Tao2} who obtained a ``relative removal lemma"
which essentially covers the case of $L_{\infty}$ regular hypergraphs. Our second main result extends the relative
removal lemma in the whole range of admissible $p$'s.
\begin{thm} \label{t2.2}
Let $n,r\in\nn$ with $n\meg r\meg 2$, and let $C\meg 1$ and $1<p\mik\infty$. Then for every $0<\ee\mik 1$ there exist two strictly
positive constants $\eta=\eta(n,r,C,p,\ee)$ and $\delta=\delta(n,r,C,p,\ee)$, and a positive integer $k=k(n,r,C,p,\ee)$ with the
following property. Let $\mathscr{H}=(n, \langle(X_i,\Sigma_i,\mu_i):i\in [n]\rangle,\calh)$ be an $\eta$-nonatomic, $r$-uniform
hypergraph system and let $\langle \nu_e:e\in\calh\rangle$ be a $(C,\eta,p)$-pseudorandom family. For every $e\in \calh$ let
$f_e\in L_1(\bbx,\calb_e,\bmu)$ with $0\mik f_e\mik \nu_e$ such that
\begin{equation} \label{e2.4}
\int \prod_{e\in\calh} f_e\, d\bmu \mik \delta.
\end{equation}
Then for every $e\in \calh$ there exists $F_e\in\calb_e$ with
\begin{equation} \label{e2.5}
\int_{\bbx\setminus F_e} f_e\, d\bmu\mik\ee \ \text{ and } \ \bigcap_{e\in \calh} F_e=\emptyset.
\end{equation}
Moreover, there exists a collection $\langle\calp_{\! e'}: e'\subseteq e \text{ for some } e\in\calh \rangle$
of partitions of\, $\bbx$ such that: \emph{(i)} $\calp_{\! e'}\subseteq \calb_{e'}$ and $|\calp_{\! e'}|\mik k$
for every $e'\subseteq e\in\calh$, and \emph{(ii)} for every $e\in\calh$ the set $F_e$ belongs to the algebra
generated by the family $\bigcup_{e'\varsubsetneq e} \calp_{\! e'}$.
\end{thm}
Here, a \emph{$(C,\eta,p)$-pseudorandom family} is a collection $\langle\nu_e:e\in\calh\rangle$ of nonnegative random
variables each of which effectively resembles a function $\psi_e\in L_p(\bbx,\calb_e,\bbs)$ with $\|\psi_e\|_{L_p}\mik C$
(in particular, each $\nu_e$ is $L_p$ regular). The precise definition is given in Section 6, along with some basic properties.
If $\psi_e=1$ for every $e\in\calh$, then this definition coincides\footnote{Actually, it is weaker than the linear forms condition,
and formally coincides with the conditions which are discussed in \cite{CFZ}.} with the ``linear forms condition" introduced in \cite{Tao2},
and so Theorem \ref{t2.2} yields the stronger version of the relative removal lemma obtained recently in \cite{CFZ}. More important,
in the general case we do not demand that the functions $\langle \psi_e:e\in\calh\rangle$ are independent and, in fact,
there are natural examples of pseudorandom families consisting of mutually correlated random variables~(see~\cite{DKK2}).

As we have already indicated, the proof of Theorem \ref{t2.2} follows the general philosophy of the transfer method of Green and Tao.
It is based on Theorem~\ref{t2.1} as well as on a relative counting lemma for pseudorandom families (Theorem \ref{t7.1}
in the main text). This relative counting lemma in turn relies on a new decomposition method which we shall discuss in detail
in Subsection 7.1. The reader is also referred to \cite{DK1} for further applications of this method in an arithmetic context.
We also remark that the problem of obtaining a counting lemma for $L_p$ regular hypergraphs was asked in \cite[Section 8]{BCCZ}.
\medskip

\noindent 2.3. \textbf{Consequences.} We conclude our discussion in this section by noting that Theorem \ref{t2.2} yields
a new Szemer\'{e}di-type theorem for sparse pseudorandom subsets of finite additive groups. We present this result in Subsection 8.3.
Readers mainly interested in the integers will also find in Subsection 8.3 a purely arithmetic version of Theorem \ref{t2.2}
(Theorem \ref{t8.3} in the main text).


\section{Background material}

\numberwithin{equation}{section}

Our general notation and terminology is standard. By $\nn=\{0,1,\dots\}$ we denote the set of all natural numbers.
\medskip

\noindent 3.1. \textbf{Martingale difference sequences.} Let $(X,\Sigma,\mu)$ be a probability space and recall that a finite
sequence $(d_i)_{i=0}^n$ of integrable real-valued random variables on $(X,\Sigma,\mu)$ is said to be a \emph{martingale
difference sequence} if there exists a martingale $(f_i)_{i=0}^n$ such that $d_0=f_0$ and $d_i=f_i-f_{i-1}$
if $n\meg 1$ and $i\in [n]$.

It is clear that every square-integrable martingale difference sequence $(d_i)_{i=0}^n$ is orthogonal in $L_2$ and, therefore,
\begin{equation} \label{e3.1}
\Big( \sum_{i=0}^n \|d_i\|_{L_2}^2 \Big)^{1/2} = \big\| \sum_{i=0}^n d_i \big\|_{L_2}.
\end{equation}
We will need the following extension of this basic fact.
\begin{prop} \label{p3.1}
Let $(X,\Sigma,\mu)$ be a probability space and $1<p \mik 2$. Then for every martingale difference
sequence $(d_i)_{i=0}^n$ in $L_p(X,\Sigma,\mu)$ we have
\begin{equation} \label{e3.2}
\Big( \sum_{i=0}^n \|d_i\|^2_{L_p} \Big)^{1/2} \mik \Big(\frac{1}{p-1}\Big)^{1/2} \, \big\| \sum_{i=0}^n d_i\big\|_{L_p}.
\end{equation}
\end{prop}
It is a remarkable fact that the constant $(p-1)^{-1/2}$ appearing in the right-hand side of \eqref{e3.2} is best possible.
This sharp estimate was recently proved by Ricard and Xu \cite{RX} who deduced it from the following uniform convexity inequality
for $L_p$ spaces (see \cite{BCL} or \cite[Lemma 4.32]{Pi}).
\begin{prop} \label{p3.2}
Let $(X,\Sigma,\mu)$ be an arbitrary measure space and $1<p \mik 2$. Then for every $x,y\in L_p(X,\Sigma,\mu)$ we have
\begin{equation} \label{e3.3}
\|x\|_{L_p}^2 + (p-1) \|y\|_{L_p}^2 \mik \frac{\|x+y\|_{L_p}^2+\|x-y\|_{L_p}^2}{2}.
\end{equation}
\end{prop}

\noindent 3.2. \textbf{Probability spaces with small atoms.} Recall that a classical theorem of Sierpi\'{n}ski asserts that for every
nonatomic finite measure space $(X,\Sigma,\mu)$ and every $0\mik c\mik \mu(X)$ there exists $C\in\Sigma$ with $\mu(C)=c$.
We will need the following approximate version of this result.
\begin{prop} \label{p3.3}
Let $0<\eta< 1$ and let $(X,\Sigma,\mu)$ be a probability space such that $\mu(A)\mik \eta$ for every atom $A$ of $(X,\Sigma,\mu)$.
Also let $B\in\Sigma$ with $\mu(B)>\eta$, and let $\eta\mik c<\mu(B)$. Then there exists $C\in\Sigma$ with $C\subseteq B$ and $c\mik \mu(C)< c+\eta$.
\end{prop}
Proposition \ref{p3.3} is straightforward for discrete probability spaces. The general case follows from the aforementioned
result of Sierpi\'{n}ski and a transfinite exhaustion argument. We give a proof for the convenience of the reader.
\begin{proof}[Proof of Proposition \emph{\ref{p3.3}}]
Assume not, that is,
\begin{enumerate}
\item[(H)] for every $C\in\Sigma$ with $C\subseteq B$ either $\mu(C)<c$ or $\mu(C)\meg c+\eta$.
\end{enumerate}
We will use hypothesis (H) to construct an increasing family $\langle Z_{\alpha}:\alpha<\omega_1\rangle$ of measurable events of $(X,\Sigma,\mu)$
such that $Z_{\alpha}\subseteq B$, $\mu(Z_{\alpha})<c$ and $\mu(Z_{\alpha+1}\setminus Z_{\alpha})>0$ for every $\alpha<\omega_1$.
Clearly, this leads to a contradiction.

We begin by setting $Z_0=\emptyset$. If $\alpha$ is a limit ordinal, then we set $Z_{\alpha}=\bigcup_{\beta<\alpha} Z_{\beta}$;
notice that $\mu(Z_{\alpha})\mik c$ and so, by hypothesis (H), we have $\mu(Z_\alpha)<c$. Finally, let $\alpha=\beta+1$
be a successor ordinal. By Sierpi\'{n}ski's result and hypothesis (H), the set $B\setminus Z_{\beta}$ must contain an atom $A$
of $(X,\Sigma,\mu)$. We set $Z_{\alpha}=Z_{\beta}\cup A$ and we observe that $\mu(Z_{\beta+1}\setminus Z_{\beta})=\mu (A)>0$.
Also notice that $\mu(Z_{\alpha})<c+\eta$. Thus, invoking hypothesis (H) once again, we conclude that
$\mu(Z_{\alpha})<c$ and the proof of Proposition \ref{p3.3} is completed.
\end{proof}

\noindent 3.3. \textbf{Removal lemma for hypergraph systems.} We will need the following version of the removal lemma for
hypergraph systems which is due to Tao \cite{Tao1} (see also \cite{DK2} for an exposition). Closely related
discrete analogues were obtained earlier by Gowers \cite{Go1} and, independently, by Nagle, R\"{o}dl, Schacht and Skokan \cite{NRS,RSk}.
\begin{thm} \label{t3.4}
For every $n,r\in\nn$ with $n\meg r\meg 2$ and every $0<\ee\mik 1$ there exist a strictly positive constant $\Delta(n,r,\ee)$ and a
positive integer $K(n,r,\ee)$ with the following property. Let $\mathscr{H}=(n,\langle(X_i,\Sigma_i,\mu_i):i\in [n]\rangle,\calh)$
be an $r$-uniform hypergraph system and for every $e\in \calh$ let $E_e\in\calb_e$ such that
\begin{equation} \label{e3.4}
\bmu \Big(\bigcap_{e\in \calh}E_e\Big) \mik \Delta(n,r,\ee).
\end{equation}
Then for every $e\in \calh$ there exists $F_e\in\calb_e$ with
\begin{equation} \label{e3.5}
\bmu(E_e\setminus F_e)\mik \ee \ \text{ and } \ \bigcap_{e\in \calh} F_e=\emptyset.
\end{equation}
Moreover, there exists a collection $\langle\calp_{\! e'}: e'\subseteq e \text{ for some } e\in\calh \rangle$
of partitions of\, $\bbx$ such that: \emph{(i)} $\calp_{\! e'}\subseteq \calb_{e'}$ and $|\calp_{\! e'}|\mik K(n,r,\ee)$
for every $e'\subseteq e\in\calh$, and \emph{(ii)} for every $e\in\calh$ the set $F_e$ belongs to the algebra
generated by the family $\bigcup_{e'\varsubsetneq e} \calp_{\! e'}$.
\end{thm}


\section{A H\"{o}lder-type inequality}

\numberwithin{equation}{section}

In this section, our goal is to prove a H\"{o}lder-type inequality for $L_p$ regular random variables. We begin with a brief
motivating discussion. Let $(X,\Sigma,\mu)$~be~a probability space and let $f$ be a nonnegative real-valued random variable
which belongs to $L_p(X,\Sigma,\mu)$ for some $1<p\mik \infty$. Let $q$ denote the conjugate exponent~of~$p$ and notice that,
by H\"{o}der's inequality, for every $A\in\Sigma$ we have
\begin{equation} \label{e4.1}
\Big( \int_A f\, d\mu\Big)^q \mik \|f\|^q_{L_p} \cdot \mu(A).
\end{equation}
Of course, this uniform estimate does not hold true if $f$ is not in $L_p(X,\Sigma,\mu)$. Nevertheless, we do have a version
of \eqref{e4.1} if we merely assume that f is $L_p$ regular. This is the content of the following proposition.
\begin{prop} \label{p4.1}
Let  $n,r \in \nn$ with $n \meg r \meg 2$ and $0<\eta\mik (r+1)^{-1}$. Also let $C>0$ and $1<p \mik \infty$, and let $q$ denote
the conjugate exponent of $p$. Finally, let $\mathscr{H}=(n,\langle(X_i,\Sigma_i,\mu_i)\colon i \in [n]\rangle, \calh)$ be an
$\eta$-nonatomic hypergraph system, $e \in \calh$ with $|e|=r$, and let $f \in L_1(\bbx,\calb_e,\boldsymbol{\mu})$ be nonnegative.
Then the following hold.
\begin{enumerate}
\item [(a)] If $f$ is $(C,\eta,p)$-regular, then for every $A \in \cals_{\partial_e}$ we have
\begin{equation} \label{e4.2}
\Big( \int_A f\, d\bmu\Big)^q \mik C^q\big(\bmu(A) + (r+3)\eta\big).
\end{equation}
\item [(b)] On the other hand, if \eqref{e4.2} is satisfied for every $A \in \cals_{\partial_e},$ then the random variable $f$ is
$(K,\eta,p)$-regular where $K= C(r+4)^{1/q}\eta^{-1/p}$. In particular, if $p=\infty$, then  $f$ is $(C(r+4),\eta,\infty)$-regular.
\end{enumerate}
\end{prop}
Proposition \ref{p4.1} is based on the simple (but quite useful) observation that for every $A\in\cals_{\partial e}$
with $\bmu(A)\meg \eta$ we can find a partition of $\bbx$ which almost contains the set $A$, and whose members are contained in
$\cals_{\partial e}$ and are not too small. We present this fact in a slightly more general form in Subsection 4.1 (this form
is needed later on in Section 5). The proof of Proposition \ref{p4.1} is completed in Subsection 4.2.
\medskip

\noindent 4.1. \textbf{A partition lemma.} For every probability space $(X,\Sigma,\mu)$ and every finite partition $\calp$
of $X$ with $\calp\subseteq \Sigma$ we set
\begin{equation} \label{e4.3}
\iota(\calp)=\min\{\mu(P): P\in\calp\}.
\end{equation}
We have the following lemma.
\begin{lem} \label{l4.2}
Let $r$ be a positive integer and $0<\theta<1$. Also let $(X,\Sigma,\mu)$ be a probability space, $(\calb_i)_{i=1}^r$ a finite
sequence of sub-$\sigma$-algebras of\, $\Sigma$, and set
\[ \cals=\Big\{\bigcap_{i=1}^r A_i: A_i\in\calb_i \text{ for every } i\in [r] \Big\}. \]
Then for every $A\in\cals$ with $\mu(A)\meg \theta$ there exist: \emph{(i)} a partition $\calq$ of $X$ with
$\calq\subseteq\cals$ and  $\iota(\calq)\meg \theta$, \emph{(ii)} a set $B\in\calq$ with $A\subseteq B$, and \emph{(iii)} pairwise disjoint
sets $B_1,\dots,B_r\in \cals$ with $\mu(B_i)<\theta$ for every $i\in [r]$, such that $B\setminus A=\bigcup_{i=1}^r B_i$.
\end{lem}
\begin{proof}
Fix $A\in\cals$ with $\mu(A)\meg \theta$ and write $A=\bigcap_{i=1}^r A_i$ where $A_i\in\calb_i$ for every $i\in [r]$.
For every nonempty $I\subseteq [r]$ and every $i\in I$ let
\[ C_{I, i}=\Big( \bigcap_{j\in \{\ell\in I: \ell <i\}} A_j \Big)\cap (X\setminus A_i) \]
with the convention that $ C_{I,i}=X\setminus A_i$ if $i=\min (I)$. It is clear that $C_{I,i}\in\cals$ for every $i\in I$.
Moreover, notice that the family $\{C_{I,i}: i\in I\}$ is a partition of $X \setminus\bigcap_{i\in I}A_i$. We set
$G=\big\{i\in [r]: \mu(C_{[r],i})\meg \theta\big\}$ and we observe that if $G=\emptyset$, then the trivial
partition $\calq=\{X\}$ and the sets $C_{[r],1},\dots,C_{[r],r}$ satisfy the requirements of the lemma. So, assume
that $G$ is nonempty and let
\[  B=\bigcap_{i\in G} A_i \ \text{ and } \ \calq=\{B\}\cup\big\{C_{G,i}: i\in G\big\}. \]
Also let $ B_i=B \cap C_{[r]\setminus G, i}$ if $i\notin G$, and $B_i=\emptyset$ if $i\in G$.
We will show that $\calq, B$ and $B_1,\dots,B_r$ are as desired.

Indeed, notice first that $\calq$ is a partition of $X$ with $\calq\subseteq \cals$, $B\in\calq$ and  $A\subseteq B$.
Next, let $Q\in\calq$ be arbitrary. If $Q=B$, then $\mu(Q)=\mu(B)\meg \mu(A)\meg \theta$. Otherwise,
there exists $i\in G$ such that $Q=C_{G,i}$. Since $C_{[r],i}\subseteq C_{G,i}$ and $i\in G$,
we see that $\mu(Q)=\mu(C_{G,i}) \meg \mu(C_{[r],i}) \meg \theta$. Thus, we have $\iota(\calq)\meg \theta$.
Finally, observe that $B_1,\dots,B_r\in\cals$ are pairwise disjoint and
\[ B\setminus A = \bigcup_{i=1}^r (B\cap C_{[r],i}) = \bigcup_{i\notin G} (B\cap C_{[r]\setminus G,i}) = \bigcup_{i=1}^r B_i. \]
Moreover, for every $i\notin G$ we have
\[ B_i=B\cap C_{[r]\setminus G,i}= \Big(\bigcap_{j\in G} A_j\Big) \cap C_{[r]\setminus G,i}\subseteq C_{[r],i} \]
and so $\mu(B_i) \mik \mu(C_{[r],i})<\theta$. The proof of Lemma \ref{l4.2} is completed.
\end{proof}

\noindent 4.2. \textbf{Proof of Proposition \ref{p4.1}.} We will need the following lemma. It is a version of Proposition \ref{p3.3}
in the context of hypergraph systems.
\begin{lem} \label{l4.3}
Let  $n,r \in \nn$ with $n\meg r \meg 2$ and $0<\alpha, \eta <1$ with $r \eta \mik 1-a$. Also let
$\mathscr{H}=(n,\langle(X_i,\Sigma_i,\mu_i)\colon i \in [n]\rangle, \calh)$ be an $\eta$-nonatomic hypergraph system,
and let $e \in \calh$ with $|e| = r$. Then for every $A\in \cals_{\partial e}$ with $\bmu(A)<a$ there exists
$B\in \cals_{\partial e}$ with $A\subseteq B$ and $a \mik \boldsymbol{\mu}(B)< a+2\eta$.
\end{lem}
\begin{proof}
We argue as in the proof of Proposition \ref{p3.3}. Specifically, fix $A\in\cals_{\partial e}$ with $\bmu(A)<a$
and assume, towards a contradiction, that
\begin{enumerate}
\item[(H)] for every $B\in\cals_{\partial e}$ with $A\subseteq B$ either $\bmu(B)<a$ or $\bmu(B)\meg a+2\eta$.
\end{enumerate}
For every $e'\in \partial e$ we select $A_{e'}\in\calb_{e'}$ such that $A=\bigcap_{e'\in\partial e} A_{e'}$ and we observe that
\begin{equation} \label{e4.4}
\sum_{e'\in\partial e} \bmu(\bbx\setminus A_{e'})\meg \bmu\Big(\bbx\setminus \bigcap_{e'\in\partial e} A_{e'}\Big)>1-a \meg r\eta.
\end{equation}
Therefore, there exists $e'_1\in \partial e$ such that $\bmu(\bbx\setminus A_{e'_1})>\eta$. Since $\mathscr{H}$ is $\eta$-nonatomic,
we see that $\bmu(Z)\mik \eta^{r-1}\mik \eta$ for every atom $Z$ of $(\bbx,\calb_{e'_1},\bmu)$. Hence, by Proposition \ref{p3.3}
applied  for ``$B=\bbx\setminus A_{e'_1}$" and ``$c=\eta$", there exists $B_{e'_1}\in\calb_{e'_1}$ with $B_{e'_1}\subseteq \bbx\setminus A_{e'_1}$
and $\eta\mik\bmu(B_{e'_1})<2\eta$. We set $A^1_{e'_1}= A_{e'_1}\cup B_{e'_1}$ and $A^1_{e'}=A_{e'}$ if $e'\in \partial e\setminus \{e'_1\}$.
Notice that: (i) $\bmu(A^1_{e'_1})\meg \bmu(A_{e'_1})+\eta$, (ii) $\bigcap_{e' \in \partial e} A^1_{e'}\in\cals_{\partial e}$, and (iii)
$A\subseteq \bigcap_{e' \in \partial e}  A^1_{e'}$. Moreover, $\bmu\big(\bigcap_{e'\in\partial e}A^1_{e'}\big)\mik\bmu(A)+2\eta< a+2\eta$
and so, by hypothesis (H), we obtain that $\bmu\big(\bigcap_{e' \in \partial e} A^1_{e'}\big)<a$. It follows, in particular, that
the estimate in \eqref{e4.4} is satisfied for the family $\langle A^1_{e'}:e' \in \partial e\rangle$.

Thus, setting $M=\lceil 2r/\eta\rceil$, we select recursively: (a) a finite sequence $(e'_m)_{m=1}^M$ in $\partial e$, and (b) for every
$e'\in\partial e$ a finite sequence $(A^m_{e'})_{m=0}^M$ in $\calb_{e'}$ with $A^0_{e'}=A_{e'}$, such that for every $m\in [M]$ the following hold.
\begin{enumerate}
\item[(C1)] For every $e' \in \partial e$ we have $A^{m-1}_{e'}\subseteq A^m_{e'}$.
Moreover, $\bmu(A^m_{e'_m}) \meg \bmu(A^{m-1}_{e'_m})+\eta$.
\item[(C2)] We have $\bmu\big(\bigcap_{e' \in \partial e} A^m_{e'}\big)<a$.
\end{enumerate}
By the classical pigeonhole principle, there exist $L\subseteq [M]$ with $|L|\meg M/r$ and $g\in\partial e$ such that $e'_m=g$ for every $m\in L$.
If $\ell=\max(L)$, then by (C1) we conclude that $\bmu(A^\ell_g)\meg 2$ which is clearly a contradiction.
\end{proof}
We are ready to give the proof of Proposition \ref{p4.1}.
\begin{proof}[Proof of Proposition \emph{\ref{p4.1}}]
(a) Fix $A\in\cals_{\partial e}$. If $\eta\mik \bmu(A)$, then we claim that
\begin{equation} \label{e4.5}
\Big( \int_A f \, d\bmu \Big)^q \mik C^q\big(\bmu(A)+r\eta\big).
\end{equation}
Indeed, by Lemma \ref{l4.2}, there exist a partition $\calq$ of $\bbx$ with $\calq\subseteq \cals_{\partial e}$ and $\iota(\calq)\meg \eta$,
and $B\in\calq$ with $A\subseteq B$ and $\bmu(B\setminus A)<r\eta$. Since $f$ is $(C,\eta,p)$-regular we see that
\[ \frac{\int_B f\, d\bmu}{\bmu(B)}\, \bmu(B)^{1/p}\mik \|\ave(f\, | \, \cala_{\calq})\|_{L_p}\mik C. \]
(Here, we have $\bmu(B)^{1/p}=1$ if $p=\infty$.) Hence,
\[ \Big( \int_A f\, d\bmu\Big)^q \mik \Big( \int_B f\, d\bmu\Big)^q \mik C^q\bmu(B) \mik C^q\big(\bmu(A)+r\eta\big). \]
Next, assume that $0\mik \bmu(A)<\eta$. Our hypothesis that $0<\eta \mik (r+1)^{-1}$ yields that $r\eta \mik 1-\eta$ and so,
by Lemma \ref{l4.3}, there exists $B\in\cals_{\partial e}$ with $A\subseteq B$ and $\eta\mik\bmu(B)<3\eta$. Therefore,
\begin{equation} \label{e4.6}
\Big( \int_A f\, d\bmu\Big)^q \mik \Big( \int_B f\, d\bmu \Big)^q \stackrel{\eqref{e4.5}}{\mik}
C^q \big(\bmu(B)+r\eta\big) \mik C^q \big(\bmu(A)+(r+3)\eta\big)
\end{equation}
and the proof of part (a) is completed.
\medskip

\noindent (b) Let $\calp$ be an arbitrary partition of $\bbx$ with $\calp\subseteq\cals_{\partial e}$ and $\iota(\calp)\meg\eta$.
By \eqref{e4.2}, for every $P\in\calp$ we have $\int_P f\, d\bmu \mik C(r+4)^{1/q} \bmu(P)^{1/q}$. Therefore, if $1 < p < \infty$,
\begin{eqnarray*}
\|\ave( f\, | \,\cala_\calp)\|_{L_p}^p & = & \sum_{P \in \calp} \Big( \frac{\int_{P} f\, d\bmu}{\bmu(P)}\Big)^p \bmu(P)
\mik C^p (r+4)^{p/q} \sum_{P \in \calp} \bmu(P)^{p/q+1-p} \\
& = & C^p (r+4)^{p/q} |\calp| \mik C^p (r+4)^{p/q} \eta^{-1}.
\end{eqnarray*}
On the other hand, if $p=\infty$,
\[ \|\ave( f\, | \, \cala_\calp)\|_{L_{\infty}} = \max\Big\{ \frac{\int_P f\, d\bmu}{\bmu(P)}: P \in \calp\Big\}
\mik \frac{C\big( \bmu(P)+(r+3)\eta\big)}{\bmu(P)} \mik C(r+4) \]
as desired.
\end{proof}


\section{Regularity lemma for $L_p$ regular hypergraphs}

\numberwithin{equation}{section}

In this section we give the proof of Theorem \ref{t2.1}. Specifically, let $n,r\in\nn$ with $n\meg r\meg 2$,
and let $C>0$ and $1<p\mik \infty$. Let $q$ denote the conjugate exponent of $p$ and set $p^{\dagger}=\min\{2,p\}$.
Also let $\mathscr{H}=(n,\langle(X_i,\Sigma_i,\mu_i):i\in [n]\rangle,\calh)$ be an $r$-uniform hypergraph system.
\emph{These data will be fixed throughout this section}.
\medskip

\noindent 5.1. \textbf{Preliminary tools.} The following result is a refinement of the partition lemma presented in Subsection 4.1.
Recall that for every probability space $(X,\Sigma,\mu)$ and every partition $\calp$ of $X$ with $\calp\subseteq \Sigma$
we set $\iota(\calp)=\min\{\mu(P):P\in\calp\}$.
\begin{lem} \label{l5.1}
Let $0<\vartheta,\eta<1$ and $e\in\calh$. Let $f\in L_1(\bbx,\calb_e,\bmu)$ be nonnegative and $(C,\eta,p)$-regular,
and $\calp$ a finite partition of\, $\bbx$ with $\calp\subseteq\cals_{\partial e}$. Assume that
\begin{equation} \label{e5.1}
\eta\mik \big(\vartheta \cdot \iota(\calp)\big)^q
\end{equation}
and that $\mathscr{H}$ is $\eta$-nonatomic. Then for every $A\in\cals_{\partial e}$ there exist:
\emph{(i)} a refinement $\calq$ of\, $\calp$ with $\calq\subseteq \cals_{\partial e}$ and
$\iota(\calq)\meg \big(\vartheta \cdot \iota(\calp)\big)^q$, and \emph{(ii)} a set $B\in\cala_{\calq}$, such that
\begin{equation} \label{e5.2}
\int_{A\triangle B} \! \ave(f\,|\, \cala_{\calp}) \, d\bmu\mik C r\vartheta  \ \text{ and } \
\int_{A\triangle B} \! f \, d\bmu \mik 5C r^2\vartheta.
\end{equation}
\end{lem}
\begin{proof}
It is based on the following greedy algorithm. Let $P$ be a cell of the partition~$\calp$ and assume that the trace $A\cap P$
of the given set $A$ on $P$ is a relatively large subset of $P$. In this case, using Lemma \ref{l4.2} applied for the conditional
probability measure with respect to $P$, we may select a partition $\calq_P$ of $P$ and an element $B_P$ of $\calq_P$ which is a
rather strong approximation of $A\cap P$. On the other hand, if the measure of $A\cap P$ is negligible when compared with the measure
of $P$, then we do not split~$P$. After scanning all the cells of the partition $\calp$, this process will eventually produce a refinement
$\calq$ of $\calp$ and a fairly accurate global approximation $B$ of $A$ which belongs to the $\sigma$-algebra generated by $\calq$.
Once this is done, the estimates in \eqref{e5.2} will follow from H\"{o}lder's inequality and Proposition \ref{p4.1} respectively.

We proceed to the details. We fix $A\in\cals_{\partial e}$ and we set
\begin{equation} \label{e5.3}
\theta=\vartheta^q \cdot\iota(\calp)^{q-1}.
\end{equation}

First, for every $P\in\calp$ we select a partition $\calq_P$ of $P$ with $\calq_P\subseteq \cals_{\partial e}$ and a set
$B_P\in \cals_{\partial e}$ as follows. Let $P\in\calp$ be arbitrary. If $\bmu(A\cap P)< \theta\bmu(P)$, then we set
$\calq_P=\{P\}$ and $B_P=\emptyset$. Otherwise, let $(P,\bbs_P,\bmu_P)$ be the probability space where
$\bbs_P=\{C\cap P: C\in \bbs\}$ and $\bmu_P$ is the conditional probability measure of $\bmu$ with respect to $P$, that is,
$\bmu_P(C)=\bmu(C\cap P)/\bmu(P)$ for every $C\in\bbs$. Write $\partial e=\{e'_1,\dots, e'_r\}$ and for every $i\in [r]$
let $\calb_i=\{B\cap P: B\in \calb_{e'_i}\}$; observe that $\calb_i$ is a sub-$\sigma$-algebra of $\bbs_P$. Also let
$\cals=\big\{\bigcap_{i=1}^r B_i: B_i\in\calb_i \text{ for every }i\in [r] \big\}\subseteq \cals_{\partial e}$.
By Lemma \ref{l4.2} applied to the probability space $(P,\bbs_P,\bmu_P)$ and the set $A\cap P\in \cals$, we obtain:
(i) a partition $\calq_P$ of $P$ with $\calq_P\subseteq \cals$ and $\iota(\calq_P)\meg \theta$, (ii) a set $B_P\in\calq_P$
with  $A\cap P\subseteq B_P$, and (iii) pairwise disjoint sets $B^P_1,\dots,B^P_r\in\cals$ with $\bmu_P(B^P_i)<\theta$
for every $i\in [r]$, such that $B_P\setminus(A\cap P)=\bigcup_{i=1}^r B^P_i$.

Next, we define
\begin{equation} \label{e5.4}
\calq=\bigcup_{P\in\calp}\calq_P \ \text{ and } \ B=\bigcup_{P\in\calp} B_P.
\end{equation}
Observe that $\calq$ is a refinement of\, $\calp$ with $\calq\subseteq \cals_{\partial e}$ and
$\iota(\calq)\meg \theta \cdot \iota(\calp) = \big( \vartheta \cdot \iota(\calp)\big)^q$.
Also note that $B\in\cala_{\calq}$ and, setting $\calp^*=\{P\in \calp: \bmu(A\cap P)\meg \theta\bmu(P)\}$, we have
\begin{equation} \label{e5.5}
A\, \triangle\, B = \Big( \bigcup_{P\in\calp\setminus \calp^*}(A\cap P)\Big) \cup
\Big( \bigcup_{P\in\calp^*}\big(\bigcup_{i=1}^{r} B^P_{i}\big) \Big)
\end{equation}
where for every $P\in\calp^*$ the sets $B_1^P,\dots,B^P_r$ are as in (iii) above. In particular, noticing that
$\bmu(A\cap P)< \theta\bmu(P)$ for every $P\notin\calp^*$ and $\bmu(B^P_i)<\theta\bmu(P)$ for every $P\in\calp^*$
and every $i\in [r]$, we see that
\begin{equation} \label{e5.6}
\bmu(A\,\triangle\, B)\mik r\theta.
\end{equation}
On the other hand, by \eqref{e5.1}, we have $\iota(\calp)\meg \eta$, and so $\|\ave(f\,|\,\cala_{\calp})\|_{L_p}\mik C$
since $f$ is $(C,\eta,p)$-regular. Hence, by H\"older's inequality, we obtain that
\[ \int_{A\triangle B} \ave(f\,|\,\cala_{\calp})\,d\bmu\mik \|\ave(f\,|\,\cala_{\calp})\|_{L_p} \cdot
\bmu(A\,\triangle\, B)^{1/q}\stackrel{\eqref{e5.6}}{\mik} C r^{1/q}\theta^{1/q} \stackrel{\eqref{e5.3}}{\mik} C r \vartheta. \]
We proceed to show that $\int_{A\triangle B} f\,d\bmu \mik  5C r^2\vartheta$. To this end notice first that
\begin{equation} \label{e5.7}
\int_{A\triangle B} f\, d\bmu = \sum_{P\in\calp\setminus\calp^*} \int_{A\cap P} f\, d\bmu +
\sum_{P\in\calp^*}\sum_{i=1}^r \int_{B^P_i} f\,d\bmu.
\end{equation}
By \eqref{e5.1} and \eqref{e5.3}, we have $\eta\mik \theta\bmu(P)$ for every $P\in\calp$. Hence, if $P\in\calp\setminus \calp^*$,
then, by Proposition \ref{p4.1},
\begin{eqnarray*}
\Big( \int_{A\cap P} f\,d\bmu\Big)^q & \mik & C^q \big(\bmu(A\cap P) +(r+3)\eta\big) \\
& \mik & C^q\big(\theta \bmu(P)+(r+3)\theta \bmu(P)\big)\mik 5 C^q r\theta \bmu(P)
\end{eqnarray*}
and so
\begin{equation} \label{e5.8}
\sum_{P\in\calp\setminus \calp^*} \int_{A\cap P} f\,d\bmu \mik 5Cr\theta^{1/q} \sum_{P\in\calp\setminus\calp^*} \bmu(P)^{1/q}.
\end{equation}
Respectively, for every $P\in\calp^*$ and every $i\in [r]$ we have
\[\Big( \int_{B^P_i} f\,d\bmu\Big)^q \mik C^q\big(\bmu(B^P_i)+(r+3)\eta\big) \mik 5C^q r\theta \bmu(P) \]
which yields that
\begin{equation} \label{e5.9}
\sum_{P\in\calp^*} \sum_{i=1}^r \int_{B^P_i} f\,d\bmu \mik 5C r^2 \theta^{1/q} \sum_{P\in\calp^*}\bmu(P)^{1/q}.
\end{equation}
Finally, notice that the function $x^{1/q}$ is concave on $\rr^+$ since $q\meg 1$. Therefore,
\begin{equation} \label{e5.10}
\sum_{P \in\calp} \bmu(P)^{1/q} \mik |\calp|^{1/p}\mik \iota(\calp)^{-1/p}.
\end{equation}
Combining \eqref{e5.7}--\eqref{e5.10} we conclude that
\[ \int_{A \triangle B} f\,d\bmu \mik 5Cr^2\theta^{1/q} \sum_{P \in \calp} \bmu(P)^{1/q} \mik
5Cr^2\theta^{1/q} \cdot \iota(\calp)^{-1/p} \stackrel{\eqref{e5.3}}{=} 5Cr^2\vartheta \]
and the proof of Lemma \ref{l5.1} is completed.
\end{proof}
Lemma \ref{l5.1} is used in the proof of the following result which asserts (roughly speaking) that if a given approximation
of an $L_p$ regular random variable $f$ is not sufficiently close to $f$ in the cut norm, then we can find a much finer approximation.
\begin{lem} \label{l5.2}
Let  $0<\delta,\eta< 1$ and set $\vartheta=\delta(12Cr^2)^{-1}$. Also let $e\in\calh$ and let $f\in L_1(\bbx,\calb_e,\bmu)$
be nonnegative and $(C,\eta,p)$-regular. Finally, let $\calp$ be a finite partition of\, $\bbx$ with $\calp\subseteq\cals_{\partial e}$
such that $\|f-\ave(f\,|\,\cala_{\calp})\|_{\cals_{\partial e}}>\delta$. Assume that
\begin{equation} \label{e5.11}
\eta\mik \big(\vartheta \cdot \iota(\calp)\big)^q
\end{equation}
and that $\mathscr{H}$ is $\eta$-nonatomic. Then there exists a refinement $\calq$ of\, $\calp$ with $\calq\subseteq \cals_{\partial e}$
and $\iota(\calq)\meg \big(\vartheta \cdot \iota(\calp)\big)^q$, such that
$\|\ave(f\,|\,\cala_{\calq})-\ave(f\,|\,\cala_{\calp})\|_{L_{p^{\dagger}}}\meg \delta/2$.
\end{lem}
\begin{proof}
We select $A\in\cals_{\partial e}$ such that
\begin{equation} \label{e5.12}
\big| \int_A \big(f-\ave(f\,|\,\cala_{\calp})\big)\, d\bmu\big|>\delta.
\end{equation}
Next, we apply Lemma \ref{l5.1} and we obtain a refinement $\calq$ of $\calp$ with $\calq\subseteq \cals_{\partial e}$
and $\iota(\calq)\meg \big(\vartheta \cdot \iota(\calp)\big)^q$, and a set $B\in\cala_{\calq}$ such that
$\int_{A\triangle B} \ave(f\,|\,\cala_{\calp}) \, d\bmu\mik C r\vartheta$ and
$\int_{A\triangle B} f\, d\bmu\mik 5C r^2\vartheta$. Then, by the choice of $\vartheta$, we have
\begin{multline*}
\ \ \ \ \big| \int_A \big(f-\ave(f\,|\,\cala_{\calp})\big)\,d\bmu - \int_B \big(f-\ave(f\,|\,\cala_{\calp})\big)\,d\bmu\big|\mik \\
\mik \int_{A\triangle B} f\,d\bmu + \int_{A\triangle B} \ave(f\,|\,\cala_{\calp})\,d\bmu
\mik 5Cr^2\vartheta + C r\vartheta \mik 6C r^2\vartheta=\delta/2
\end{multline*}
and so, by \eqref{e5.12},
\begin{equation} \label{e5.13}
\big|\int_B \big(f-\ave(f\,|\,\cala_{\calp})\big)\,d\bmu\big|\meg\delta/2.
\end{equation}
On the other hand, the fact that $B\in\cala_{\calq}$ yields that
\begin{equation} \label{e5.14}
\int_B \big(f-\ave(f\,|\,\cala_{\calp})\big)\, d\bmu=
\int_B \big(\ave(f\,|\,\cala_{\calq})-\ave(f\,|\,\cala_{\calp})\big)\, d\bmu.
\end{equation}
Therefore, by the monotonicity of the $L_p$ norms, we conclude that
\begin{multline*}
\ \ \ \ \|\ave(f\,|\,\cala_{\calq})-\ave(f\,|\,\cala_{\calp})\|_{L_{p^\dagger}} \meg
\|\ave(f\,|\,\cala_{\calq})-\ave(f\,|\,\cala_{\calp})\|_{L_1} \meg \\
\big|\int_B \big(\ave(f\,|\,\cala_{\calq})-\ave(f\,|\,\cala_{\calp})\big)\, d\bmu\big|
\stackrel{\eqref{e5.14}}{=} \big|\int_B \big(f-\ave(f\,|\,\cala_{\calp})\big)\, d\bmu\big|
\stackrel{\eqref{e5.13}}{\meg} \delta/2
\end{multline*}
as desired.
\end{proof}
\noindent 5.2. \textbf{Proof of Theorem \ref{t2.1}.} We now enter into the main part of the proof. We begin with the following lemma.
\begin{lem} \label{l5.3}
Let $0<\delta,\eta<1$ and $0<\sigma\mik 1$, and set $\vartheta=\delta(12Cr^2)^{-1}$ and $N=\lceil 4 ({p^{\dagger}}-1)^{-1}\sigma^2 \delta^{-2}\rceil$.
Also let $e\in\calh$, let $f\in L_1(\bbx,\calb_e,\boldsymbol{\mu})$ be nonnegative and $(C,\eta,p)$-regular, and let $\calp$ be a finite partition
of\, $\bbx$ with $\calp\subseteq\cals_{\partial e}$. Assume that
\begin{equation} \label{e5.15}
\eta\mik \big(\vartheta^{N} \cdot \iota(\calp)\big)^{q^N}
\end{equation}
and that $\mathscr{H}$ is $\eta$-nonatomic. Then there exists a refinement $\calq$ of\, $\calp$ with $\calq\subseteq\cals_{\partial e}$
and $\iota(\calq) \meg \big(\vartheta^{N} \cdot \iota(\calp)\big)^{q^N}$\!, such that either
\begin{enumerate}
\item[(a)] $\|\ave(f\,|\,\cala_{\calq})-\ave(f\,|\,\cala_{\calp})\|_{L_{p^{\dagger}}}>\sigma$, or
\item[(b)] $\|\ave(f\,|\,\cala_{\calq})-\ave(f\,|\,\cala_{\calp})\|_{L_{p^{\dagger}}}\mik\sigma$ and
$\|f-\ave(f\,|\,\cala_{\calq})\|_{\cals_{\partial e}}\mik\delta$.
\end{enumerate}
\end{lem}
The case ``$p\meg 2$" in Lemma \ref{l5.3} can be proved with an ``energy increment strategy"\!, a method which originates from the
work of Szemer\'{e}di \cite{Sz1,Sz2} and is essentially based upon the fact that martingale difference sequences are orthogonal in $L_2$
(see, e.g., \cite{DK2,Tao1}). The general case follows a method introduced recently~in~\cite{DKK1} and relies on Proposition \ref{p3.1}
(see also \cite{DKT} for another application).
\begin{proof}[Proof of Lemma \emph{\ref{l5.3}}]
Assume that part (a) is not satisfied, that is,
\begin{enumerate}
\item[(H1)] for every refinement $\calq$ of $\calp$ with $\calq\subseteq\cals_{\partial e}$
and $\iota(\calq)\meg\big(\vartheta^{N}\cdot\iota(\calp)\big)^{q^N}$ we have
$\|\ave(f\,|\,\cala_{\calq})-\ave(f\,|\,\cala_{\calp})\|_{L_{p^{\dagger}}}\mik\sigma$.
\end{enumerate}
We claim that there exists a refinement $\calq$ of $\calp$ which satisfies the second part of the lemma. Indeed, if not, then,
by (H1) and Lemma \ref{l5.2}, we see that
\begin{enumerate}
\item[(H2)] for every refinement $\calq$ of $\calp$ with $\calq\subseteq\cals_{\partial e}$ and
$\iota(\calq)\meg\big(\vartheta^{N}\cdot \iota(\calp)\big)^{q^N}$  there exists a refinement $ \calr$ of $\calq$
with $\calr\subseteq \cals_{\partial e}$ and $\iota(\calr) \meg\big(\vartheta \cdot \iota(\calq)\big)^q$ such that
$\|(f\,|\,\cala_{\calr})-\ave(f\,|\,\cala_{\calq})\|_{L_{p^{\dagger}}}>\delta/2$.
\end{enumerate}
Recursively and using (H2), we select partitions $\calp_0,\dots,\calp_N$ of $\bbx$ with $\calp_0=\calp$
such that for every $i\in [N]$ we have: (P1) $\calp_i$ is a refinement of $\calp_{i-1}$ with $\calp_i\subseteq\cals_{\partial e}$
and $\iota(\calp_i)\meg \big(\vartheta \cdot \iota(\calp_{i-1})\big)^q$, and (P2)
$\|\ave(f\,|\,\cala_{\calp_i})-\ave(f\,|\,\cala_{\calp_{i-1}})\|_{L_{p^{\dagger}}}>\delta/2$.

Next, set $g=f-\ave(f\,|\,\cala_{\calp})$ and let $(d_i)_{i=0}^N$ be the difference sequence associated with the finite
martingale $\ave(g\,|\,\cala_{\calp_0}),\dots,\ave(g\,|\,\cala_{\calp_N})$. Notice that for every $i \in [N]$ we have
$d_i=\ave(f\,|\,\cala_{\calp_i})-\ave(f\,|\,\cala_{\calp_{i-1}})$ which implies, by (P2), that $\|d_i\|_{L_{p^{\dagger}}}>\delta/2$.
Therefore, by the choice of $N$ and Proposition \ref{p3.1},
\begin{eqnarray*}
\sigma \mik (p^{\dagger}-1)^{1/2}\,\frac{\delta}{2}\, N^{1/2} & < & (p^{\dagger}-1)^{1/2} \Big( \sum_{i=0}^N \|d_i\|_{L_{p^{\dagger}}}^2 \Big)^{1/2} \\
& \mik & \big\|\sum_{i=0}^N d_i \big\|_{L_{p^{\dagger}}}= \| \ave(f\,|\,\cala_{\calp_N})-\ave(f\,|\,\cala_{\calp})\|_{L_{p^{\dagger}}}.
\end{eqnarray*}
On the other hand, by (P1), we see that $\calp_N$ is a refinement of $\calp$ with $\calp_N\subseteq\cals_{\partial e}$
and $\iota(\calq)\meg\big(\vartheta^{N}\cdot\iota(\calp)\big)^{q^N}$ and so, by (H1),
$\|\ave(f\,|\,\cala_{\calp_N})-\ave(f\,|\,\cala_{\calp})\|_{L_{p^{\dagger}}} \mik \sigma$ which contradicts, of course, the above estimate.
The proof is thus completed.
\end{proof}
We introduce some numerical invariants. For every growth function $F\colon \nn\to\rr$ and every $0<\sigma\mik 1$ we define, recursively,
a sequence $(N_m)$ in $\nn$ and two sequences $(\eta_m)$ and $(\vartheta_m)$ in $(0,1]$ by setting $N_0=0$, $\eta_0=1$,
$\theta_0=\big(12Cr^2F(1)\big)^{-1}$ and
\begin{equation} \label{e5.16}
\begin{cases}
N_{m+1}=\lceil 4 (p^{\dagger}-1)^{-1} \sigma^2 F(\lceil\eta_m^{-1}\rceil)^2\rceil, \\
\eta_{m+1}= (\vartheta_m^{N_{m+1}} \cdot \eta_m)^{q^{N_{m+1}}}, \\
\vartheta_{m+1}=\big( 12Cr^2 F(\lceil\eta_{m+1}^{-1}\rceil)\big)^{-1}.
\end{cases}
\end{equation}
The following lemma is the last step of the proof of Theorem \ref{t2.1}.
\begin{lem} \label{l5.4}
Let  $0<\sigma\mik 1$ and $F\colon\nn\to \rr$ a growth function. Set
\begin{equation} \label{e5.17}
L=\lceil C^2(p^{\dagger}-1)^{-1} \sigma^{-2}n^r\rceil
\end{equation}
and let $(\eta_m)$ be as in \eqref{e5.16}. Let\, $0<\eta\mik \eta_{L}$ and assume that $\mathscr{H}$ is $\eta$-nonatomic.
For every $e\in\calh$ let $f_e\in L_1(\bbx,\calb_e,\boldsymbol{\mu})$ be nonnegative and $(C,\eta,p)$-regular. Then there exist:
\emph{(i)} a positive integer $m\in\{0,\dots,L-1\}$, \emph{(ii)} for every $e\in\calh$ a partition $\calp_{\!e}$ of\, $\bbx$ with
$\calp_{\!e}\subseteq \cals_{\partial e}$ and $\iota(\calp_{\!e})\meg \eta_m$, and \emph{(iii)} for every $e\in\calh$ a refinement
$\calq_e$ of\, $\calp_{\!e}$ with $\calq_e\subseteq \cals_{\partial e}$ and $\iota(\calq_e)\meg \eta_{m+1}$, such that for every
$e\in\calh$ we have $\|\ave(f_e\,|\,\cala_{\calq_e})-\ave(f_e\,|\,\cala_{\calp_{\!e}})\|_{L_{p^\dagger}}\mik\sigma$ and
$\|f_e-\ave(f_e\,|\,\cala_{\calq_e})\|_{\cals_{\partial e}}\mik 1/F(\lceil\eta_m^{-1}\rceil)$.
\end{lem}
\begin{proof}
It is similar to the proof of Lemma \ref{l5.3} and so we will briefly sketch the argument. If the lemma is false,
then using Lemma \ref{l5.3} we select for every $e\in\calh$ partitions $\calp_0^e,\dots,\calp_{L}^e$ of $\bbx$ with
$\calp_0^e=\{\bbx\}$ as well as $e_1,\dots,e_L\in\calh$ such that for every $m\in [L]$ we have: (P1) $\calp_m^e$ is a refinement
of $\calp_{m-1}^e$ with $\calp_{m}^e\subseteq \cals_{\partial e}$ and $\iota(\calp_m^e)\meg \eta_{m}$ for every $e\in\calh$,
and (P2) $\|\ave(f_{e_m}\,|\,\cala_{\calp_m^{e_m}})-\ave(f_{e_m}\,|\,\cala_{\calp_{m-1}^{e_m}})\|_{L_{p^{\dagger}}}>\sigma$.
By the pigeonhole principle, there exist $\boldsymbol{e}\in\calh$ and $I\subseteq [L]$ with $|I|\meg L/n^r$, such that
$\boldsymbol{e}=e_m$ for every $m\in I$. Let $(d_m)_{m=0}^L$ be the difference sequence associated with the finite martingale
$\ave(f_{\boldsymbol{e}}\,|\,\cala_{\calp_0^{\boldsymbol{e}}}),\dots,\ave(f_{\boldsymbol{e}}\,|\,\cala_{\calp_L^{\boldsymbol{e}}})$
and notice that $\|d_m\|_{L_{p^{\dagger}}}>\sigma$ for every $m\in I$. Moreover, since $f_{\boldsymbol{e}}$ is
$(C,\eta,p)$-regular and $\iota(\calp_m^{\boldsymbol{e}})\meg \eta_m\meg \eta_L\meg \eta$ we see that
$\|\ave(f_{\boldsymbol{e}}\,|\,\cala_{\calp_m^{\boldsymbol{e}}})\|_{L_{p^{\dagger}}}\mik
\|\ave(f_{\boldsymbol{e}}\,|\,\cala_{\calp_m^{\boldsymbol{e}}})\|_{L_p}\mik C$ for every $m\in [L]$. Hence, by the choice of
$L$ in \eqref{e5.17} and Proposition \ref{p3.1}, we conclude that
\[ C < (p^{\dagger}-1)^{1/2}\Big( \sum_{m=0}^L \|d_m\|^2_{L_{p^{\dagger}}} \Big)^{1/2} \mik \big\|\sum_{m=0}^L d_m \big\|_{L_{p^{\dagger}}}
= \|\ave(f_{\boldsymbol{e}}\,|\,\cala_{\calp_L^{\boldsymbol{e}}})\|_{L_{p^{\dagger}}} \mik C \]
which is clearly a contradiction.
\end{proof}
We are ready to complete the proof of Theorem \ref{t2.1}.
\begin{proof}[Proof of Theorem \emph{\ref{t2.1}}]
Let $F\colon\nn\to \rr$ be a growth function and $0<\sigma\mik 1$, and let $L$ and $\eta_L$ be as in \eqref{e5.17} and \eqref{e5.16}
respectively. We set $\mathrm{Reg}=\lceil\eta_L^{-1}\rceil$ and we claim that with this choice the result follows. Indeed, set
$\eta\coloneqq 1/\mathrm{Reg}\mik \eta_L$ and assume that $\mathscr{H}$ is $\eta$-nonatomic. For every $e\in\calh$ let
$f_e\in L_1(\bbx,\calb_e,\bmu)$ be nonnegative and $(C,\eta,p)\text{-regular}$. Let $m\in\{0,\dots,L-1\}$,
$\langle \calp_{\!e}:e\in\calh\rangle$ and $\langle \calq_e:e\in\calh\rangle$ be as in Lemma \ref{l5.4} and define
$M=\lceil\eta_m^{-1}\rceil$. It is clear that $M$, $\langle \calp_{\!e}:e\in\calh\rangle$ and $\langle \calq_e:e\in\calh\rangle$
are as desired.
\end{proof}


\section{Pseudorandom families}

\numberwithin{equation}{section}

Throughout this section let $n,r$ be two positive integers with $n\meg r\meg 2$ and let
$\mathscr{H}=(n,\langle(X_i,\Sigma_i,\mu_i):i\in [n]\rangle,\calh)$ be an $r$-uniform hypergraph system.
\medskip

\noindent 6.1. \textbf{Definitions and basic properties.} We begin by introducing some pieces of notation.
Let $e\subseteq [n]$ be nonempty and recall that $\pi_e\colon\bbx\to\bbx_e$ stands for the natural projection. Observe
that for every $f\in L_1(\bbx,\calb_e,\bmu)$ there exists a unique random variable $\bfu\in L_1(\bbx_e,\bbs_e,\bmu_e)$ such that
\begin{equation} \label{e6.1}
f=\bfu\circ \pi_e
\end{equation}
and note that the map $L_1(\bbx,\calb_e,\bmu)\ni f\mapsto \bfu\in L_1(\bbx_e,\bbs_e,\bmu_e)$ is a linear isometry.
If, in addition, the set $[n]\setminus e$ is nonempty, then for every $\bx_e\in\bbx_e$ and every
$\bx_{[n]\setminus e}\in\bbx_{[n]\setminus e}$ by $(\bx_e,\bx_{[n]\setminus e})$ we denote the unique element
$\bx$ of $\bbx$ such that $\bx_e=\pi_e(\bx)$ and $\bx_{[n]\setminus e}=\pi_{[n]\setminus e}(\bx)$. Finally,
for every $f\colon \bbx\to\rr$ and every $\bx_e\in\bbx_e$ let $f_{\bx_e}\colon\bbx_{[n]\setminus e}\to \rr$ be
the section of $f$ at $\bx_e$, that is, $f_{\bx_e}(\bx_{[n]\setminus e})=f(\bx_e,\bx_{[n]\setminus e})$.
\begin{defn} \label{d6.1}
Let $C\meg 1$ and\, $0<\eta<1$. Also let $1<p\mik\infty$ and let $q$ denote the conjugate exponent of $p$.
For every $e\in\calh$ let $\nu_e\in L_1(\bbx,\calb_e,\bmu)$ be a nonnegative random variable. We say that the family
$\langle \nu_e:e\in\calh\rangle$ is \emph{$(C,\eta,p)$-pseudorandom} provided that the following conditions are satisfied.
\begin{enumerate}
\item [(C1)] \emph{(Copies of sub-hypergraphs of $\calh$)} For every nonempty $\calg\subseteq\calh$ we have
\begin{equation} \label{e6.2}
\int \prod_{e\in\calg} \nu_e\,d\bmu\meg 1-\eta.
\end{equation}
\item [(C2)] For every $e\in\calh$ there exists $\psi_e\in L_p(\bbx,\calb_e,\bmu)$ with $\|\psi_e\|_{L_p} \mik C$
and satisfying the following properties.
\begin{enumerate}
\item[(a)] \emph{(The cut norm of $\nu_e-\psi_e$ is negligible)} We have $\|\nu_e-\psi_e\|_{\cals_{\partial e}}\mik \eta$.
\item[(b)] \emph{(Local linear forms)} For every $j\in\{0,1\}$ and every $e'\in\calh\setminus\{e\}$ let
$g_{e'}^{(j)}\in L_1(\bbx,\calb_{e'},\bmu)$ such that either $0 \mik g_{e'}^{(j)} \mik \nu_{e'}$ or\, $0 \mik g_{e'}^{(j)}\mik 1$.
If\, $\boldsymbol{\nu}_e$ and $\boldsymbol{\psi}_e$ are as in \eqref{e6.1} for $\nu_e$ and $\psi_e$ respectively, then
\begin{equation} \label{e6.3}
\big| \int (\boldsymbol{\nu}_e-\boldsymbol{\psi}_e)(\bx_e) \prod_{j\in\{0,1\}} \Big(\int \prod_{e'\in\calh\setminus\{e\}}
(g_{e'}^{(j)})_{\bx_e}\,d\bmu_{[n]\setminus e}\Big)\,  d\bmu_e\big| \mik \eta.
\end{equation}
\end{enumerate}
\item[(C3)] \emph{(Integrability of marginals)} Let $e\in\calh$ and let $\calg\subseteq\calh\setminus \{e\}$ be nonempty,
and define $\boldsymbol{\nu}_{e,\calg}\colon\bbx_e \to \rr$ by\,
$\boldsymbol{\nu}_{e,\calg}(\bx_e)=\int\prod_{e'\in\calg}(\nu_{e'})_{\bx_e}\, d\bmu_{[n]\setminus e}$. Then, setting
\begin{equation} \label{e6.4}
\ell \coloneqq \min\Big\{2n :n\in\nn \text{ and } 2n \meg 2q + \Big(1 -\frac{1}{C }\Big) + \frac{1}{p}\Big\}
\end{equation}
$($where $1/p=0$ if $p=\infty$$)$, we have
\begin{equation} \label{e6.5}
\int \boldsymbol{\nu}_{e,\calg}^{\ell}\, d\bmu_e \mik C+\eta.
\end{equation}
\end{enumerate}
\end{defn}
As we have noted, examples of pseudorandom families are presented in \cite{DKK2}. Here, as an illustration, we will describe
a class of these examples which are quite easy to grasp. (The verification of their properties, although not particularly difficult,
is somewhat lengthy and is given in \cite{DKK2}.) They are of the form $\langle \nu_e+\varphi_e:e\in\calh\rangle \text{ where}$:
(i) $\calh$ is the $r$-simplex, that is, the complete $r$-uniform hypergraph on $[r+1]\text{, (ii) each}$ $\nu_e\in L_1(\bbx,\calb_e,\bmu)$
is nonnegative and the family $\langle \nu_e:e\in\calh\rangle$ satisfies a slight modification of the ``linear forms condition"
introduced in \cite[Definition 2.8]{Tao2}, and (iii) each $\varphi_e\in L_p(\bbx,\calb_e,\bmu)$ is also nonnegative and satisfies
$\|\boldsymbol{\varphi}_e^p\|_{\square_{\ell}(\bbx_e)}\mik C^p$ where $\ell$ is as in \eqref{e6.4} and $\|\cdot\|_{\square_{\ell}(\bbx_e)}$
is a variant of the Gowers box norm \cite{Go1} which was first considered by Hatami \cite{Ha}. In other words, the family
$\langle \nu_e+\varphi_e:e\in\calh\rangle$ is a perturbation of a system of measures which appears in \cite{Tao2}, the main
point being that only integrability conditions are imposed on each ``noise" $\varphi_e$.

We now proceed to discuss some properties of pseudorandom families.
\medskip

\noindent 6.1.1. First observe that condition (C1) expresses a natural combinatorial requirement, namely that the weighted hypergraph
$\langle \nu_e:e\in\calh\rangle$ contains many copies of every sub-hypergraph of $\calh$.
\medskip

\noindent 6.1.2. Condition (C2.a) is also rather mild and implies that each $\nu_e$ is, to some extend, well-behaved.
Specifically, we have the following lemma.
\begin{lem} \label{l6.2}
If the family $\langle\nu_e:e\in\calh\rangle$ satisfies condition \emph{(C2.a)}, then for every $e\in \calh$
the random variable $\nu_e$ is $(C+1,\eta,p)$-regular.
\end{lem}
\begin{proof}
Let $e\in\calh$ and let $\calp$ be a partition of $\bbx$ with $\calp\subseteq\cals_{\partial e}$ and $\bmu(P)\meg\eta$
for every $P\in\calp$. By condition (C2.a), for every $P\in\calp$ we have
\[\frac{|\int_P (\nu_e - \psi_e) \,d\bmu|}{\bmu(P)}\mik 1\]
and, consequently, $\|\ave(\nu_e-\psi_e\,|\,\cala_{\calp})\|_{L_\infty}\mik 1$. Therefore, by the triangle inequality
and the monotonicity of the $L_p$ norms, we conclude that
\[ \|\ave(\nu_e\,|\,\cala_{\calp})\|_{L_p} \mik \|\ave(\psi_e\,|\,\cala_{\calp})\|_{L_p}+
\|\ave(\nu_e-\psi_e\,|\,\cala_{\calp})\|_{L_p} \mik C+1 \]
and the proof is completed.
\end{proof}
Condition (C2.b), the local linear forms condition, is the strongest (and as such, the most restrictive) condition of all.
In the case where $\psi_e=1$ for every $e\in\calh$ it was explicitly isolated\footnote{We remark that in \cite{CFZ} condition (C2.b)
is referred to as the ``strong linear forms" condition.} in \cite[Lemma 6.3]{CFZ}, though closely related variants had appeared
earlier in the work of Green and Tao \cite{GT1}. One of the signs of the strength of the local linear forms condition is that it
implies condition (C2.a) as long as the hypergraph $\calh$ is not too sparse. More precisely, assume that for every $e\in\calh$
we have $\partial e\subseteq\{e'\cap e: e'\in\calh\}$ (this is the case, for instance, if $\calh$ is the $r$-simplex).
Fix~$e\in\calh$ and for every $f\in\partial e$ let $A_f\in\calb_f$. We set $g_{e'}^{(0)}=\mathbf{1}_{A_{f}}$ if $e'\cap e=f$;
otherwise, let $g_{e'}^{(j)}=1$. By \eqref{e6.3}, we see that
$|\int (\nu_e-\psi_e)\prod_{f \in \partial e}\mathbf{1}_{A_{f}}\,d\bmu |\mik\eta$ which implies, of course, that
$\|\nu_e-\psi_e\|_{\cals_{\partial e}}\mik \eta$.
\medskip

\noindent 6.1.3. Condition (C3) can be seen as an instance of the general fact that by taking averages we improve integrability.
In a nutshell, we demand that the marginal $\boldsymbol{\nu}_{e,\calg}$ belongs to $L_\ell(\bbx_e,\bbs_e,\bmu_e)$ where $\ell$
is an even integer strictly bigger than~$2q$, though in the special case $C=1$ and $p=\infty$ we can simply take $\ell=2$.
The exact choice of $\ell$ in \eqref{e6.4} expresses these requirements in a formal way, but the reader may actually think
of\footnote{Note that this simplification is not far from the truth. Indeed, if $q$ is an integer and $C>1$, then, by \eqref{e6.4},
we have that $\ell=2(q+1)$. It is also useful to point out that in the proofs of Theorems \ref{t2.2} and \ref{t7.1} (which are our
main results concerning pseudorandom families), we will not use the fact that the number $\ell$ in \eqref{e6.4} is an even integer.
However, this property is needed when one is working with the Gowers box norms and their variants (see, e.g.,
\cite[Proposition 2.1]{DKK2}), and so it is reasonable to add this harmless condition in Definition \ref{d6.1}.} the parameter
$\ell$ as being equal to $2(q+1)$.

The following lemma reduces condition (C3) to certain moment estimates. We point out that these estimates are needed
for the analysis of pseudorandom families.
\begin{lem} \label{l6.3}
If the family $\langle\nu_e:e\in\calh\rangle$ satisfies condition \emph{(C3)}, then for every $e\in\calh$ and every
nonempty $\calg\subseteq\calh\setminus\{e\}$ the following hold.
\begin{enumerate}
\item[(a)] If either $C>1$ or $1<p<\infty$, then $\ell>2q$ and for every $\bfca\in\bbs_e$ we have
\begin{equation} \label{e6.6}
\int_\bfca \boldsymbol{\nu}_{e,\calg}^{2q}\, d\bmu_e \mik (C+1)\,\bmu_{e}(\bfca)^{e(C,p)}
\end{equation}
where $e(C,p)=(4pq)^{-1}$ if\, $1<p<\infty$, and $e(C,\infty)=1/2$ if $C>1$.
\item[(b)] Assume that the family $\langle\nu_e:e\in\calh\rangle$ also satisfies condition \emph{(C1)},
and that $C=1$ and $p=\infty$. Then $\ell =2$ and $\|\bnu_{e,\calg} - 1\|^2_{L_2} \mik 4\eta^{1/2}$.
In particular, for every $\bfca\in\bbs_e$ we have
\begin{equation} \label{e6.7}
\int_\bfca \boldsymbol{\nu}_{e,\calg}^2\,d\bmu_e \mik 2\bmu_e(\bfca) + 8\eta^{1/2}.
\end{equation}
\end{enumerate}
\end{lem}
\begin{proof}
(a) The fact that $\ell> 2q$ follows immediately by \eqref{e6.4}. Next, fix $\bfca\in\bbs_e$. By  H\"older's inequality, we have
\begin{equation} \label{e6.8}
\int_\bfca \boldsymbol{\nu}_{e,\calg}^{2q}\,d \bmu_e \mik \|\boldsymbol{\nu}_{e,\calg}\|_{L_\ell}^{2q} \cdot
\bmu_e(\bfca)^{1-\frac{2q}{\ell}} \stackrel{\eqref{e6.5}}{\mik} (C+1)\,\bmu_e(\bfca)^{1-\frac{2q}{\ell}}.
\end{equation}
On the other hand, by \eqref{e6.4} and the choice of $e(C,p)$, we see that $1-\frac{2q}{\ell}\meg e(C,p)$.
By \eqref{e6.8}, the proof of part (a) is completed.
\medskip

\noindent (b) First observe that $\ell =2$. Moreover, by Fubini's theorem and Jensen's inequality,
\[ 1-\eta \stackrel{\mathrm{(C1)}}{\mik} \int\prod_{e'\in\calg}\nu_{e'}\,d\bmu =\int \boldsymbol{\nu}_{e,\calg}\,d\bmu_e \mik
\Big( \int \boldsymbol{\nu}_{e,\calg}^2\,d\bmu_e \Big)^{1/2} \stackrel{\mathrm{(C3)}}{\mik} (1+\eta)^{1/2} \]
and, consequently, $|\int(\boldsymbol{\nu}_{e,\calg}^2-1)\, d\bmu_e|\mik 2\eta$\, and
$|\int (\boldsymbol{\nu}_{e,\calg}-1)\,d \bmu_e| \mik \eta^{1/2}$. Therefore,
\begin{eqnarray} \label{e6.9}
\|\bnu_{e,\calg} - 1\|_{L_2}^2 & = & \int (\bnu_{e,\calg}^2 - 2\,\bnu_{e,\calg} +1)\,d\bmu_e \\
& \mik & \big|\int(\bnu_{e,\calg}^2 -1)\,d\bmu_e\big| + 2\,\big|\int (\bnu_{e,\calg} -1)\,d\bmu_e\big| \mik 4\eta^{1/2}. \nonumber
\end{eqnarray}
Now let $\bfca\in\bbs_e$ and note that $\|\bnu_{e,\calg}\cdot\mathbf{1}_\bfca-\mathbf{1}_\bfca\|_{L_2}\mik \|\bnu_{e,\calg}- 1\|_{L_2}$.
Hence, by \eqref{e6.9} and the triangle inequality, we have $\|\bnu_{e,\calg}\cdot \mathbf{1}_\bfca\|_{L_2} \mik \|\mathbf{1}_\bfca\|_{L_2} +
(4\eta^{1/2})^{1/2}$ and so
\[ \int_\bfca \bnu_{e,\calg}^2 \,d\bmu_e\mik \big( \bmu_e(\bfca)^{1/2} + (4\eta^{1/2})^{1/2}\big)^2\mik 2\bmu_e(\bfca) + 8\eta^{1/2} \]
as desired.
\end{proof}

\noindent 6.2. \textbf{Conditions on the majorants.} Although for the analysis of pseudorandom families we need precisely conditions
(C1)--(C3), in practice some of these conditions are not so easily checked. This is the case, for instance, with the local linear forms
condition, since it requires verifying the estimate in \eqref{e6.3} not only for the ``majorants" $\langle \nu_e:e\in\calh\rangle$ but
also for all nonnegative functions which are pointwise bounded by them. However, this problem can be effectively resolved by imposing
some slightly stronger conditions on $\langle \nu_e:e\in\calh\rangle$, and then reducing \eqref{e6.3} to these conditions by repeated
applications of the Cauchy--Schwarz inequality. This method was developed extensively by Green and Tao \cite{GT1,GT2} and has become
standard in the field. As such, our presentation in this subsection will be rather brief.

First we need to introduce some notation. Recall that by $n,r$ we denote two positive integers with $n\meg r\meg 2$
and $\mathscr{H}=(n,\langle(X_i,\Sigma_i,\mu_i):i\in [n]\rangle,\calh)$ stands for an $r$-uniform hypergraph system.
For every $m\in\nn$ with $m\meg 2$ and every nonempty $e\subseteq [n]$ let $(\bbx_e^m,\bbs_e^m,\bmu_e^m)$ be the product
of $m$ copies of the probability space $(\bbx_e,\bbs_e,\bmu_e)$. (If $e=[n]$, then this product will be denoted by
$(\bbx^m,\bbs^m,\bmu^m)$.) Next, let $h\subseteq [n]$ with $e\subseteq h$. For every
$\bx_h^{(0)}=(x_i^{(0)})_{i\in h},\dots,\bx_h^{(m -1)}=(x_i^{(m-1)})_{i\in h}$ in $\bbx_h$ and every
$\omega=(\omega_i)_{i\in e}\in\{0,\ldots,m-1\}^e$ we set
\[ \bx_e^{(\omega)}=(x_i^{(\omega_i)})_{i \in e}\in\bbx_e. \]
Note that if $\omega=j^e$ for some $j\in\{0,\dots,m-1\}$ (that is, $\omega=(\omega_i)_{i\in e}$ with $\omega_i=j$ for every $i\in e$),
then $\bx_e^{(\omega)}=\pi_{h,e}(\bx_h^{(j)})$ where $\pi_{h,e}\colon \bbx_h\to\bbx_e$ is the natural projection.
\begin{prop} \label{p6.4}
Let $C\meg 1$ and\, $0<\eta<1$. Also let $1<p\mik\infty$ and let $q$ denote the conjugate exponent of $p$.
For every $e\in\calh$ let $\nu_e \in L_1(\bbx,\calb_e,\bmu)$ be a nonnegative random variable and let $\bnu_e$ be as in \eqref{e6.1}
for $\nu_e$. Assume that the following properties are satisfied.
\begin{enumerate}
\item[(P1)] If $\ell$ is as in \eqref{e6.4}, then
\[ 1-\eta \mik \int \prod_{e\in\calh} \prod_{\omega \in\{0,\dots,\ell-1\}^e}\bnu_e^{n_{e,\omega}}(\bx_e^{(\omega)})\, d\bmu^{\ell}\mik C+\eta \]
for any choice of\, $n_{e,\omega}\in\{0,1\}$.
\item[(P2)] For every $e\in\calh$ there exists $\psi_e\in L_p(\bbx,\calb_e, \bmu)$ with $\|\psi_e\|_{L_p}\mik C$ such that
\[ \big|\int\prod_{\omega\in\{0,1\}^{e}} (\bnu_e-\boldsymbol{\psi}_e)(\bx_e^{(\omega)})
\prod_{e'\in\calh\setminus\{e\}} \prod_{\omega\in\{0,1\}^{e'}} \bnu_{e'}^{n_{e'\!,\omega}}(\bx_{e'}^{(\omega)})\, d\bmu^2\big| \mik \eta \]
for any choice of\, $n_{e'\!,\omega}\in\{0,1\}$.
\end{enumerate}
Then $\langle \nu_e:e \in\calh\rangle$ is a $(C,\eta',p)$-pseudorandom family where $\eta'=(C+1)\,\eta^{1/2^r}$.
\end{prop}
For comparison, it is useful at this point to recall that a family $\langle \nu_e:e\in\calh\rangle$ satisfies the ``linear forms condition"
in the sense of \cite[Definition 2.8]{Tao2} if we have
\begin{equation} \label{e6.10}
\int\prod_{e\in\calh} \prod_{\omega \in\{0,1\}^e}\bnu_e^{n_{e,\omega}}(\bx_e^{(\omega)})\, d\bmu^{2}=1+o(1)
\end{equation}
for any choice of\, $n_{e,\omega}\in\{0,1\}$. (Here, the quantity $o(1)$ is an error term that can be made arbitrarily small.)
Thus, we see that if $C=1$, $p=\infty$ and $\psi_e=1$ for every $e\in\calh$, then properties (P1) and (P2) in
Proposition \ref{p6.4} are equivalent to \eqref{e6.10}.

For the proof of Proposition \ref{p6.4} we need the following lemmas. The first one is a straightforward consequence of the
Gowers--Cauchy--Schwarz inequality (see, e.g., \cite[Lemma B.2]{GT2}).
\begin{lem} \label{l6.5}
Let $e\in\calh$ and $f\in L_1(\bbx,\calb_e,\bmu)$. If\, $\mathbf{f}$ is as in \eqref{e6.1} for $f$, then
\[ \|f\|_{\cals_{\partial e}}\mik\Big(\int \prod_{\omega\in\{0,1\}^e} \mathbf{f}(\bx_e^{(\omega)})\, d\bmu_e^2\Big)^{1/2^r}. \]
\end{lem}
The second lemma follows arguing precisely as in the proof of \cite[Lemma 6.3]{CFZ} or \cite[Proposition 5.1]{Tao2}.
We leave the details to the interested reader.
\begin{lem} \label{l6.6}
Let $C\meg 1$, $0<\eta<1$ and $1<p\mik\infty$. For every $e \in \calh$ let $\nu_e \in L_1(\bbx,\calb_e,\bmu)$ be a nonnegative
random variable and assume that the family $\langle \nu_e: e\in \calh\rangle$ satisfies properties \emph{(P1)} and \emph{(P2)}
in Proposition \emph{\ref{p6.4}}. Fix $e\in\calh$, and for every $j\in\{0,1\}$ and every $e'\in\calh\setminus\{e\}$ let
$g_{e'}^{(j)}\in L_1(\bbx,\calb_{e'},\bmu)$ such that either $0\mik g_{e'}^{(j)}\mik\nu_{e'}$ or\, $0\mik g_{e'}^{(j)}\mik 1$. Then we have
\[ \big|\int (\bnu_e-\boldsymbol{\psi}_e)(\bx_e) \prod_{j \in \{0,1\}} \Big(\int \prod_{e'\in\calh\setminus \{e\}}(g_{e'}^{(j)})_{\bx_e}
\, d\bmu_{[n]\setminus e}\Big)\, d\bmu_e\big| \mik (C+1)\, \eta^{1/2^r}. \]
That is, condition \emph{(C2.b)} is satisfied for the parameter $(C+1)\, \eta^{1/2^r}$.
\end{lem}
We are ready to give the proof of Proposition \ref{p6.4}.
\begin{proof}[Proof of Proposition \emph{\ref{p6.4}}]
Let $\calg\subseteq\calh$ be nonempty and observe that
\[ \int\prod_{e\in\calg} \nu_e\,d\bmu=\int\prod_{e\in\calh}\prod_{\omega\in\{0,\dots,\ell-1\}^e}
\bnu_e ^{n_{e,\omega}}(\bx_e^{(\omega)})\,d\bmu^{\ell} \]
where $n_{e,\omega}\in\{0,1\}$ and $n_{e,\omega}=1$ if and only if $e\in\calg$ and $\omega=0^e$. Hence, by property (P1),
condition (C1) is satisfied.

Next, let $e\in\calh$ by arbitrary. By property (P2), we have
\[ \int\prod_{\omega\in\{0,1\}^e}(\bnu_e-\boldsymbol{\psi}_e)(\bx_e^{(\omega)})\,d\bmu_e^2\mik \eta \]
and so, by Lemma \ref{l6.5}, we see that condition (C2.a) is satisfied. Condition (C2.b) follows, of course, from Lemma \ref{l6.6}.

It remains to verify condition (C3). To this end, fix $e\in\calh$ and let $\calg$ be a nonempty subset of $\calh\setminus \{e\}$. Then,
\begin{eqnarray*}
\int\bnu_{e,\calg}^{\ell}\,d\bmu_e & = & \int\Big(\int\prod_{e'\in\calg}\bnu_{e'}(\bx_{e'\cap e},\bx_{e'\setminus e})\,
d\bmu_{[n] \setminus e}\Big)^{\ell}\,d\bmu_{e} \\
& = & \int\Big(\int\prod_{e'\in\calg}\prod_{j=0}^{\ell-1}\bnu_{e'}(\bx_{e' \cap e},\bx_{e'\setminus e}^{(j^{e'\setminus e})})\,
d\bmu_{[n]\setminus e}^{\ell}\Big)\,d\bmu_e \\
& = & \int\prod_{e'\in\calg}\prod_{\omega\in\{0,\dots,\ell-1\}^{e'}}\bnu_{e'}^{n_{e'\!,\omega}}(\bx_{e'}^{(\omega)})\,d\bmu^{\ell}
\end{eqnarray*}
where $n_{e'\!,\omega}\in\{0,1\}$ and $n_{e'\!,\omega}=1$ if and only if $\omega\in\{(0^{e'\cap e},j^{e'\setminus e}):0\mik j\mik \ell-1\}$.
By property (P1), we conclude that condition (C3) is satisfied, and the proof of Proposition \ref{p6.4} is completed.
\end{proof}


\section{Relative counting lemma for pseudorandom families}

\numberwithin{equation}{section}

The following theorem is the main result in this section.
\begin{thm} \label{t7.1}
Let $n,r\in\nn$ with $n\meg r\meg 2$, and let $C\meg 1$ and $1<p\mik\infty$. Also let $\zeta\meg 1$ and $0<\gamma\mik 1$. Then there exist
two strictly positive constants $\eta=\eta(n,r,C,p,\zeta,\gamma)$ and $\alpha=\alpha(n,r,C,p,\zeta,\gamma)$ with the following property.
Let $\mathscr{H}=(n,\langle (X_i,\Sigma_i,\mu_i)\colon i\in [n]\rangle,\calh)$ be an $r$-uniform hypergraph system, and let
$\langle \nu_e:e\in\calh\rangle$ be a $(C,\eta,p)$-pseudorandom family. Moreover, for every $e \in \calh$ let $g_e,h_e\in L_1(\bbx,\calb_e,\bmu)$
such that $0\mik g_e\mik\nu_e$, $0\mik h_e \mik\zeta$ and $\|g_e-h_e\|_{\cals_{\partial e}}\mik\alpha$. Then we have
\begin{equation} \label{e7.1}
\big| \int\prod_{e\in\calh} g_e\, d\bmu - \int \prod_{e\in\calh} h_e\, d\bmu\big| \mik \gamma.
\end{equation}
\end{thm}
The hypotheses of Theorem \ref{t7.1} might appear rather strong: on the one hand the function $g_e$ is dominated by $\nu_e$
(and so, by Lemma \ref{l6.2}, it is $L_p$ regular), but on the other hand it is approximated in the cut norm by a nonnegative
function $h_e$ with $\|h_e\|_{L_{\infty}}\mik \zeta$. It turns out, however, that for \emph{every} $0\mik f_e\mik\nu_e$ we can
indeed satisfy these requirements by slightly truncating $f_e$. More precisely, we have the following proposition.
\begin{prop} \label{p7.2}
Let $n,r,C,p$ and $\mathscr{H}$ be as in Theorem \emph{\ref{t7.1}}, and let $M$ be a positive integer, $0<\alpha\mik 1$ and $e\in\calh$.
Also let $\calp_{\!e}$ be a partition of\, $\bbx$ with $\calp_{\!e}\subseteq \cals_{\partial e}$ and $\bmu(P)\meg 1/M$ for every
$P\in\calp_{\!e}$, and let $\calq_{e}$ be a finite refinement of\, $\calp_{\!e}$ with $\calq_{e}\subseteq \cals_{\partial e}$. Finally,
let $f_e\in L_1(\bbx,\calb_e,\bmu)$ be nonnegative and write $f_e=f^e_{\mathrm{str}}+f^e_{\mathrm{err}}+f^e_{\mathrm{unf}}$ where
$f^e_{\mathrm{str}}$, $f^e_{\mathrm{err}}$ and $f^e_{\mathrm{unf}}$ are as in \eqref{e2.2}.  Assume that the estimates in \eqref{e2.3}
are satisfied for $\sigma=\alpha/2$ and a growth function $F\colon\nn \to \rr$ with $F(m)\meg 2\alpha^{-1}m$ for every $m \in \nn$.
Then the following hold.
\begin{enumerate}
\item[(a)] For every $A\in\cala_{\calp_e}$ we have $\|f_e\cdot\mathbf{1}_A-f^e_{\mathrm{str}}\cdot\mathbf{1}_A\|_{\cals_{\partial e}}\mik\alpha$.
\item[(b)] Assume that $1<p<\infty$. Let $\zeta\meg 1$ and set $A=[f_{\mathrm{str}}^e \mik \zeta]$. Then we have $A\in\cala_{\calp_e}$ and
$\bmu(\bbx\setminus A) \mik (C/\zeta)^p$. Moreover,
\begin{equation} \label{e7.2}
\int_{\bbx\setminus A} f_e\,d\bmu \mik C^p\zeta^{1-p}+\alpha \ \text{ and } \
\int_{\bbx\setminus A} f^e_{\mathrm{str}}\,d\bmu\mik C^p\zeta^{1-p}.
\end{equation}
\end{enumerate}
\end{prop}
The rest of this section is devoted to the proof of Theorem \ref{t7.1}. As we have mentioned in the introduction, it is based on a
decomposition method which is discussed in Subsection 7.1. We also note that Proposition \ref{p7.2} is not used in the argument,
though it is needed for the proof of Theorem \ref{t2.2}. As such, we defer its proof to Subsection 8.1.
\medskip

\noindent 7.1. \textbf{Proof of Theorem \ref{t7.1}.} First we need to do some preparatory work. Let $n,r\in\nn$ with $n\meg r\meg 2$,
and let $\mathscr{H}=(n,\langle (X_i,\Sigma_i,\mu_i)\colon i\in [n]\rangle,\calh)$ be an $r$-uniform hypergraph system. Also let
$C\meg 1$ and $1<p\mik\infty$, and denote by $q$ the conjugate exponent of $p$. \textit{These data will be fixed throughout the proof.}

Next, observe that it suffices to prove Theorem \ref{t7.1} only for the case ``$\zeta=1$". Indeed, if the numbers
$\eta(n,r,C,p,1,\gamma)$ and $\alpha(n,r,C,p,1,\gamma)$ have been determined, then it is easy to see that for every $\zeta\meg 1$
Theorem \ref{t7.1} holds true for the parameters $\eta(n,r,C,p,1,\gamma \zeta^{-n^r})$ and
$\zeta \cdot \alpha(n,r,C,p,1,\gamma\zeta^{-n^r})$. Thus, in what follows we will assume that $\zeta=1$.
To avoid trivialities, we will also assume that $|\calh|\meg 2$.

We proceed to introduce some numerical invariants. For every $0<\gamma\mik 1$ we set
\begin{equation} \label{e7.3}
\beta(\gamma)=\big(10(C+1)^2\gamma^{-1}\big)^{2q/x(C,p)} \ \text{ and } \ \theta(\gamma)=\big(20(C+1)\beta(\gamma)\big)^{-2q}\gamma^{2q}
\end{equation}
where $x(C,p)=(4pq)^{-1}$ if $1<p<\infty$, $x(C,\infty)=1/2$ if $C>1$, and $x(1,\infty)=1$. Moreover, for every $m\in\{0,\dots,n^r\}$
and every $0<\gamma\mik 1$ we define $\alpha_m(\gamma)$ and $\eta_m(\gamma)$ in $(0,1]$ recursively by the rule
\begin{equation} \label{e7.4}
\alpha_0(\gamma)=\gamma/5 \ \text{ and } \ \alpha_{m+1}(\gamma)=\alpha_m\big(\theta(\gamma)\big)
\end{equation}
and
\begin{equation} \label{e7.5}
\eta_0(\gamma)=\big(30(C+1)\big)^{-4q}\gamma^{4q}  \ \text{ and } \ \eta_{m+1}(\gamma)=\eta_m\big(\theta(\gamma)\big).
\end{equation}
Notice that $\alpha_{m+1}(\gamma)\mik \alpha_m(\gamma)$ and $\eta_{m+1}(\gamma)\mik \eta_m(\gamma)$ for every $0<\gamma\mik 1$.

After this preliminary discussion we are ready to enter into the main part of the proof which proceeds by induction.
Specifically, let $\langle \nu_e:e\in\calh\rangle$ be a family of nonnegative random variables such that
$\nu_e\in L_1(\bbx,\calb_e,\bmu)$ for every $e\in\calh$. By induction on $m\in\{0,\dots,|\calh|\}$ we will show
that for every $0<\gamma\mik 1$ if the family $\langle \nu_e:e\in\calh\rangle$ is $(C,\eta_m(\gamma),p)$-pseudorandom
where $\eta_m(\gamma)$ is as in \eqref{e7.5}, then the estimate \eqref{e7.1} is satisfied for any collection
$\langle g_e,h_e\in L_1(\bbx,\calb_e,\bmu):e\in\calh\rangle$ with the following properties: (P1) for every $e\in\calh$
we have that either $0\mik g_e \mik \nu_e$ or $g_e=h_e$, (P2) for every $e\in\calh$ we have $0\mik h_e \mik 1$ and
 $\|g_e-h_e\|_{\cals_{\partial e}} \mik \alpha_m(\gamma)$ where $\alpha_m(\gamma)$ is as in \eqref{e7.4},
and (P3) $|\{e\in\calh: g_e\neq h_e\}|\mik m$.

The initial case ``$m=0$" is straightforward, and so let $m\in\{1,\dots,|\calh|\}$ and assume that the induction
has been carried out up to $m-1$. Fix $0<\gamma\mik 1$ and let $\langle g_e,h_e\in L_1(\bbx,\calb_e,\bmu):e\in\calh\rangle$
be a collection satisfying  properties (P1)--(P3). Set
\[ \Delta\coloneqq \int\prod_{e\in\calh} g_e\, d\bmu - \int \prod_{e\in\calh} h_e\, d\bmu \]
and recall that we need to show that $|\Delta|\mik\gamma$. To this end, we may assume that $|\{e\in\calh: g_e\neq h_e\}|=m$
(otherwise, the desired estimate follows immediately from the inductive assumptions). Thus, we may select $e_0\in\calh$
with $g_{e_0}\neq h_{e_0}$; note that, by property (P1), we have $0\mik g_{e_0}\mik \nu_{e_0}$. We set
$\calg=\{e\in\calh\setminus \{e_0\}: g_e\neq h_e\}$ and we define $G,H\colon \bbx_{e_0}\to\rr$ by the rule
\[ G(\bx_{e_0})=\int\!\!\!\prod_{e\in\calh\setminus\{e_0\}}\!\!\!(g_{e})_{\bx_{e_0}}\,d\bmu_{[n]\setminus e_0} \ \text{ and } \
H(\bx_{e_0})=\int\!\!\!\prod_{e\in\calh\setminus\{e_0\}}\!\!\!(h_{e})_{\bx_{e_0}} \, d\bmu_{[n]\setminus e_0}.\]
Observe that $0\mik H\mik 1$. Moreover, if $\calg$ is nonempty, then we have $0\mik G\mik\boldsymbol{\nu}_{e_0,\calg}$
where $\boldsymbol{\nu}_{e_0,\calg}$ is as in Definition \ref{d6.1}. On the other hand, notice that $G=H$ if $\calg=\emptyset$.
In the following claim we obtain a first estimate for $|\Delta|$.
\begin{claim} \label{c7.3}
We have
\begin{equation} \label{e7.6}
|\Delta|\mik 2(C+1)\big(\|G-H\|_{L_{2q}}+\eta_m(\gamma)^{1/2}\big) +\|g_{e_0}-h_{e_0}\|_{\cals_{\partial e_0}}.
\end{equation}
\end{claim}
The proof of Claim \ref{c7.3} is given in Subsection 7.2. It is important to note that it does not
use the inductive assumptions and relies, instead, on the local linear forms condition (condition (C2.b)
in Definition \ref{d6.1}) and H\"{o}lder's inequality. Closely related estimates appear in \cite{CFZ,Tao2}.

The next claim is the second step of the proof.
\begin{claim} \label{c7.4}
If $\beta(\gamma)$ and $\theta(\gamma)$ are as in \emph{\eqref{e7.3}}, then we have
\begin{equation} \label{e7.7}
\int (G-H)^{2q}\, d\bmu_{e_0}\mik 2\beta(\gamma)^{2q}\theta(\gamma)+(C+1)^2\beta(\gamma)^{-x(C,p)}+ 8\eta_m(\gamma)^{1/2}.
\end{equation}
\end{claim}
Granting Claims \ref{c7.3} and \ref{c7.4}, the proof of the inductive step (and, consequently, of Theorem \ref{t7.1}) is completed
as follows. First observe that, by \eqref{e7.5}, we have $\eta_m(\gamma)\mik \big(30(C+1)\big)^{-4q}\gamma^{4q}$; in particular,
$8\eta_m(\gamma)^{1/2}\mik \big(10(C+1)\big)^{-2q}\gamma^{2q}$. On the other hand, by Claim \ref{c7.4} and the choice of $\beta(\gamma)$
and $\theta(\gamma)$ in \eqref{e7.3}, it is easy to see that $\|G-H\|_{L_{2q}}\mik 3\big(10(C+1)\big)^{-1}\gamma$. Therefore, by Claim \ref{c7.3}
and property (P2),
\begin{eqnarray*}
|\Delta| & \mik  & 2(C+1)\big(\|G-H\|_{L_{2q}} + \eta_m(\gamma)^{1/2}\big) +\|g_{e_0} - h_{e_0}\|_{\cals_{\partial e_0}}\\
& \mik & 4\gamma/5 + \alpha_{m}(\gamma) \mik 4\gamma/5 + \alpha_{0}(\gamma) \mik 4\gamma/5 + \gamma/5=\gamma.
\end{eqnarray*}

We close this subsection by commenting on the proof of Claim \ref{c7.4} (the details are given in Subsection 7.3).
Estimates of this form are usually obtained for stronger norms than the cut norm, and as such, they depend on stronger pseudorandomness
conditions. The first general method available in this context was developed recently in \cite{CFZ}. It is known as ``densification" and
consists of taking successive marginals in order to arrive at an expression which involves only bounded functions (see also \cite{Sh}).
Another approach (which ultimately relies on H\"{o}lder's inequality) was introduced by Tao and Ziegler \cite{TZ2,TZ3}.

We propose a new method to deal with these types of problems. It appears to be quite versatile---see, e.g., \cite{DK1}---and is based
on a simple decomposition scheme. The method is best seen in action: we first observe the pointwise bound
\[ (G-H)^{2q}\mik (G-H)^{2q}\,\mathbf{1}_{[G\meg H]} + (H-G)\, H^{2q-1}\mathbf{1}_{[G<H]}. \]
Since $0\mik H^{2q-1}\mathbf{1}_{[G<H]}\mik 1$ the expectation of the second term of the above decomposition can be estimated
using our inductive hypotheses. For the first term we select a cut-off parameter $\beta \meg 1$ and we decompose further as
\[ (G-H)^{2q}\,\mathbf{1}_{[G\meg H]} \mik G^{2q}\,\mathbf{1}_{[G\meg H]}\mathbf{1}_{[G>\beta]}+
(G-H)\, G^{2q-1}\mathbf{1}_{[G\meg H]}\mathbf{1}_{[G\mik \beta]}. \]
If $\beta$ is large enough, then we can effectively bound the expectation of the first term of the new decomposition
using Lemma \ref{l6.3} and Markov's inequality. On the other hand, we have
$0\mik G^{2q-1}\mathbf{1}_{[G\meg H]}\mathbf{1}_{[G\mik \beta]}\mik\beta^{2q-1}$ and so the second term can also
be handled by our inductive assumptions. By optimizing the parameter $\beta$, we obtain the estimate in \eqref{e7.7}
thus completing the proof of Claim \ref{c7.4}.
\medskip

\noindent 7.2. \textbf{Proof of Claim \ref{c7.3}.} Let $\mathbf{g}_{e_0}$ be as in \eqref{e6.1} for $g_{e_0}$. Set
\[ I_1=\int \mathbf{g}_{e_0}(G-H)\,d\bmu_{e_0} \ \text{ and } \
I_2=\int (g_{e_0} -h_{e_0})\!\!\prod_{e\in\calh\setminus\{e_0\}}\!\! h_{e}\, d\bmu \]
and notice that $|\Delta|\mik |I_1|+|I_2|$. Next, observe that
\begin{equation} \label{e7.8}
|I_2|\mik \|g_{e_0}-h_{e_0}\|_{\cals_{\partial e_0}}.
\end{equation}
This follows by Fubini's theorem and the following well-known fact (see, e.g., \cite{Go1}).
We recall the proof for the convenience of the reader.
\begin{fact} \label{f7.5}
Let $e\in\calh$ with $|e|\meg 2$ and $g_e\in L_1(\bbx,\calb_e,\bmu)$. For every $f\in\partial e$ let $u_f\in L_\infty(\bbx,\calb_f,\bmu)$
with $0\mik u_f\mik 1$. Then we have $|\int g_e\prod_{f\in \partial e}u_f\,d\bmu|\mik\|g_e\|_{\cals_{\partial e}}$.
\end{fact}
\begin{proof} Set $k=|e|$ and let $\{f_1,\dots,f_k\}$ be an enumeration of $\partial e$. We define
$Z\colon [0,1]^k\to \mathbb{R}$ by the rule $Z(t_1,\dots,t_k)=\int g_e\prod_{i=1}^k \boldsymbol{1}_{[u_{f_i}>t_i]}\,d\bmu$.
Notice that $\bigcap_{i=1}^k [u_{f_i}>t_i]\in\cals_{\partial e}$ for every $(t_1,\dots,t_k)\in [0,1]^k$ and so
$\|Z\|_{L_\infty}\mik \|g_e\|_{\cals_{\partial e}}$. On the other hand, denoting by $\boldsymbol{\lambda}$ the
Lebesgue measure on $[0,1]^k$, by Fubini's theorem we have $\int g_e\prod_{f\in \partial e}u_f\,d\bmu=
\int Z\,d\boldsymbol{\lambda}$ and the result follows.
\end{proof}
We proceed to estimate $|I_1|$. First, by the Cauchy--Schwarz inequality and the fact that $0\mik g_{e_0}\mik \nu_{e_0}$, we obtain
\[ |I_1|^2\mik \int \mathbf{g}_{e_0}\, d\bmu_{e_0} \cdot \int \mathbf{g}_{e_0}(G-H)^2\,d\bmu_{e_0} \mik
\int\bnu_{e_0}\,d\bmu_{e_0} \cdot \int \bnu_{e_0} (G-H)^2\, d\bmu_{e_0}.\]
Let $\psi_{e_0}\in L_p(\bbx,\calb_{e_0},\bmu)$ with $\|\psi_{e_0}\|_{L_p}\mik C$ be as in Definition \ref{d6.1} and notice
that by condition (C2.a) we have $|\int (\nu_{e_0}-\psi_{e_0})\, d\bmu| \mik\eta_m(\gamma)$. This is easily seen to imply that
$\int\nu_{e_0}\,d\bmu\mik C+1$ and so, by the previous estimate, we have
\[ |I_1|^2 \mik (C+1)\cdot \Big( \big| \int \boldsymbol{\psi}_{e_0}(G-H)^2\,d\bmu_{e_0} \big| +
\big| \int (\bnu_{e_0}-\boldsymbol{\psi}_{e_0}) (G-H)^2\,d\bmu_{e_0} \big| \Big) \]
where $\boldsymbol{\psi}_{e_0}$ is as in \eqref{e6.1} for $\psi_{e_0}$. Next, writing $(G-H)^2=G^2-2GH+H^2$
and applying \eqref{e6.3}, we see that $|\int (\bnu_{e_0}-\boldsymbol{\psi}_{e_0}) (G-H)^2\,d\bmu_{e_0}|\mik 4\eta_m(\gamma)$.
On the other hand, by H\"{o}lder's inequality, $|\int\boldsymbol{\psi}_{e_0}(G-H)^2\,d\bmu_{e_0}|\mik C\|G-H\|_{L_{2q}}^2$.
Therefore,
\begin{equation} \label{e7.9}
|I_1|\mik 2(C+1)\big(\|G-H\|_{L_{2q}} +\eta_m(\gamma)^{1/2}\big).
\end{equation}
Combining \eqref{e7.8} and \eqref{e7.9} we conclude that the estimate in \eqref{e7.6} is satisfied, as desired.
\medskip

\noindent 7.3. \textbf{Proof of Claim \ref{c7.4}.} Recall that $\calg$ stands for the set $\{e\in\calh\setminus\{e_0\}:g_e\neq h_e\}$.
We may assume, of course, that $\calg$ is nonempty and, consequently, that $G\neq H$. Set $\mathbf{A}=[G<H]$,
$\mathbf{B}=[G\meg H]\cap [G\mik \beta(\gamma)]$ and $\mathbf{C}=[G\meg H]\cap [G>\beta(\gamma)]$, and notice that
$\mathbf{A},\mathbf{B},\mathbf{C}\in\bbs_{e_0}$. Next, define
\[ I_1=\int (H-G)H^{2q-1}\mathbf{1}_{\mathbf{A}}\,d\bmu_{e_0}, \ I_2=\int (G-H)\,G^{2q-1}\mathbf{1}_{\mathbf{B}}\,d\bmu_{e_0}, \
I_3=\int_{\mathbf{C}} G^{2q}\,d\bmu_{e_0} \]
and observe that $I_1,I_2,I_3\meg 0$ and $\int (G-H)^{2q}\, d\bmu_{e_0}\mik I_1+I_2+I_3$. Thus, it suffices to estimate $I_1$, $I_2$ and $I_3$.

First we argue for $I_1$. Let $h'_{e_0}=(H^{2q-1}\boldsymbol{1}_{\mathbf{A}})\circ \pi_{e_0}\in L_1(\bbx,\calb_{e_0},\bmu)$ and notice that
$0\mik h_{e_0}'\mik 1$. By the definition of $G$ and $H$, we see that
\[ I_1=\big|\int\prod_{e\in\calh\setminus\{e_0\}}g_e\cdot h'_{e_0}\,d\bmu -\int\prod_{e\in\calh\setminus\{e_0\}}h_e\cdot h'_{e_0}\,d\bmu\big|. \]
Moreover, by \eqref{e7.4} and property (P2), we have $\|g_e-h_e\|_{\cals_{\partial e}}\mik \alpha_{m-1}\big(\theta(\gamma)\big)$ for
every $e\in\calh\setminus\{e_0\}$. Since $\eta_m(\gamma)=\eta_{m-1}\big(\theta(\gamma)\big)$, by our inductive assumptions, we obtain
\begin{equation} \label{e7.10}
I_1\mik \theta(\gamma).
\end{equation}
The estimation of $I_2$ is similar. Indeed, observe that
\[ I_2 =\beta(\gamma)^{2q-1} \int (G-H)\big(G/\beta(\gamma)\big)^{2q-1}\mathbf{1}_{\mathbf{B}}\, d\bmu_{e_0} \]
and $0\mik \big(G/\beta(\gamma)\big)^{2q-1}\mathbf{1}_{\mathbf{B}}\mik 1$. Therefore,
\begin{equation} \label{e7.11}
I_2\mik \beta(\gamma)^{2q-1}\theta(\gamma).
\end{equation}
We proceed to estimate $I_3$. Let $\bnu_{e_0,\calg}$ and $\ell$ be as in Definition \ref{d6.1}, and recall that $0\mik G\mik \bnu_{e_0,\calg}$.
By Markov's inequality and the monotonicity of the $L_p$ norms,
\[ \bmu_{e_0}(\mathbf{C})\mik \bmu_{e_0}\big( [\bnu_{e_0,\calg}\meg\beta(\gamma)]\big) \mik
\frac{\int \bnu_{e_0,\calg}\, d\bmu_{e_0}}{\beta(\gamma)} \mik
\frac{\|\bnu_{e_0,\calg}\|_{L_\ell}}{\beta(\gamma)}\stackrel{\eqref{e6.5}}{\mik} \frac{C+1}{\beta(\gamma)}.\]
Thus, by Lemma \ref{l6.3} and the choice of $x(C,p)$, we have
\begin{eqnarray} \label{e7.12}
I_3 & \mik & \int_{\mathbf{C}} \bnu^{2q}_{e_0,\calg}\,d\bmu_{e_0} \mik (C+1)\, \bmu_{e_0}\!(\mathbf{C})^{x(C,p)} + 8\eta_m(\gamma)^{1/2} \\
& \mik & (C+1)^2\beta(\gamma)^{-x(C,p)}+ 8\eta_m(\gamma)^{1/2}. \nonumber
\end{eqnarray}
Combining \eqref{e7.10}--\eqref{e7.12} we conclude that the estimate in \eqref{e7.7} is satisfied. The proof of Claim \ref{c7.4} is completed.


\section{Relative removal lemma for pseudorandom families}

\numberwithin{equation}{section}

This section is devoted to the proof of Theorem \ref{t2.2} and its consequences. As we have mentioned, in Subsection 8.1
we give the proof of Proposition \ref{p7.2}. The proof of Theorem \ref{t2.2} is then completed in Subsection 8.2.
Finally, in Subsection 8.3 we present arithmetic versions of Theorem \ref{t2.2}.
\medskip

\noindent 8.1. \textbf{Proof of Proposition \ref{p7.2}.} For part (a), fix $A\in\cala_{\calp_{\!e}}$ and let
$\calp'\subseteq\calp_{\!e}$ such that $A=\bigcup \calp'$. Notice that $|\calp'|\mik |\calp_{\!e}|\mik M$ and
\[ f_e\cdot\mathbf{1}_A - f_{\mathrm{str}}^e\cdot\mathbf{1}_A= f_{\mathrm{err}}^e\cdot\mathbf{1}_A +
\sum_{P\in\calp'} f_{\mathrm{unf}}^e\cdot\mathbf{1}_P. \]
Therefore, for any $B\in\cals_{\partial_e}$ we have
\begin{eqnarray*}
\big|\int_B (f_e\cdot\mathbf{1}_A-f_{\mathrm{str}}^e\cdot\mathbf{1}_A)\,d\bmu\big| & \mik &
\big|\int_{B\cap A} f_{\mathrm{err}}^e\,d\bmu\big|  + \sum_{P\in\calp'}\big|\int_{B\cap P} f_{\mathrm{unf}}^e\,d\bmu\big| \\
& \mik & \|f^e_{\mathrm{err}}\|_{L_{p^{\dagger}}}+ |\calp'|\cdot\|f^e_{\mathrm{unf}}\|_{\cals_{\partial e}}\mik\sigma+\frac{M}{F(M)}\mik\alpha
\end{eqnarray*}
which implies, of course, that $\|f_e\cdot\mathbf{1}_A-f_{\mathrm{str}}^e\cdot\mathbf{1}_A\|_{\cals_{\partial e}}\mik \alpha$.

For part (b), let $\zeta\meg 1$ be arbitrary and set $A=[f_{\mathrm{str}}^e \mik \zeta]$. First observe that $A\in\cala_{\calp_{\!e}}$
since $f^e_{\mathrm{str}}=\ave(f_e\,|\,\cala_{\calp_{\!e}})$. Next, by Markov's inequality, we have
\[ \bmu(\bbx\setminus A) \mik \frac{\int (f_{\mathrm{str}}^e)^p\, d\bmu}{\zeta^p} \mik C^p\zeta^{-p} \]
and so, by H\"{o}lder's inequality,
\[ \int_{\bbx\setminus A} f^e_{\mathrm{str}}\,d\bmu \mik \|f^e_{\mathrm{str}}\|_{L_p}\cdot \bmu(\bbx\setminus A)^{1/q}\mik C^p\zeta^{1-p}.\]
Finally, by part (a) and the fact that $\bbx\setminus A\in\cala_{\calp_{\!e}}$, we conclude that
\begin{eqnarray*}
\int_{\bbx\setminus A} f_e\,d\bmu & \mik & \int_{\bbx\setminus A} f^e_{\mathrm{str}}\,d\bmu +
\big|\int_{\bbx\setminus A} (f_e-f_{\mathrm{str}}^e)\,d\bmu\big| \\
& \mik & C^p\zeta^{1-p}+ \|f_e\cdot \mathbf{1}_{\bbx\setminus A}-f_{\mathrm{str}}^e\cdot\mathbf{1}_{\bbx\setminus A}\|_{\cals_{\partial e}}
\mik C^p\zeta^{1-p} +\alpha
\end{eqnarray*}
and the proof of Proposition \ref{p7.2} is completed.
\medskip

\noindent 8.2. \textbf{Proof of Theorem \ref{t2.2}.} Let $n,r,C,p$ and $\ee$ be as in the statement of the theorem and let $q$ denote
the conjugate exponent of $p$. We begin by introducing some numerical invariants. First, we set
\[ \zeta=\zeta(C,p,\ee)=(C+1)^q (\varepsilon/6)^{1-q}. \]
Also let $\Delta(n,r,\frac{\ee}{6\zeta})$ and $K(n,r,\frac{\ee}{6\zeta})$ be as in Theorem \ref{t3.4} and note that we may assume that
$\Delta(n,r,\frac{\ee}{6\zeta})\mik \frac{\ee}{6\zeta}$. We define
\begin{equation} \label{e8.1}
\delta=\delta(n,r,C,p,\ee)=\frac{\Delta(n,r,\frac{\ee}{6\zeta})^{n^r}\!\!\!}{2} \ \text{ and } \
k=k(n,r,C,p,\ee)=K\Big(n,r,\frac{\ee}{6\zeta}\Big).
\end{equation}
Next, let $\alpha(n,r,C,p,\zeta,\delta)$ and $\eta(n,r,C,p,\zeta,\delta)$ be as in Theorem \ref{t7.1} and set
\[ \alpha=\min\{ k^{-2^r}\!(\ee/3),\alpha(n,r,C,p,\zeta,\delta)\} \ \text{ and } \ \mathrm{Reg}=\mathrm{Reg}(n,r,C+1,p,F,\alpha/2) \]
where $F\colon\nn\to \rr$ is the growth function defined by the rule $F(m)=2\alpha^{-1}(m+1)$ and
$\mathrm{Reg}(n,r,C+1,p,F,\alpha/2)$ is as in Theorem \ref{t2.1}. Finally, we define
\begin{equation} \label{e8.2}
\eta=\eta(n,r,C,p,\ee)=\min\{1/\mathrm{Reg},\eta(n,r,C,p,\zeta,\delta)\}.
\end{equation}
We will show that the parameters $\eta$, $\delta$ and $k$ are as desired.

Indeed, let $\mathscr{H}=(n,\langle(X_i,\Sigma_i,\mu_i):i\in [n]\rangle,\calh)$ be an $\eta$-nonatomic, $r$-uniform
hypergraph system and let $\langle \nu_e:e\in\calh\rangle$ be a $(C,\eta,p)$-pseudorandom family. For every $e\in \calh$ let
$f_e\in L_1(\bbx,\calb_e,\bmu)$ with $0\mik f_e\mik \nu_e$ and assume that
\begin{equation} \label{e8.3}
\int \prod_{e\in\calh} f_e\, d\bmu \mik \delta.
\end{equation}
By Lemma \ref{l6.2}, for every $e\in\calh$ the random variable $\nu_e$ is $(C+1,\eta,p)$-regular and, consequently, so is $f_e$.
Therefore, by \eqref{e8.2}, we may apply Theorem \ref{t2.1} and we obtain: (a) a positive integer $M$ with $M\mik \mathrm{Reg}$,
(b) for every $e\in\calh$ a partition $\calp_{\!e}$ of $\bbx$ with $\calp_{\!e}\subseteq \cals_{\partial e}$ and $\bmu(A)\meg 1/M$
for every $A\in\calp_{\!e}$, and (c) for every $e\in\calh$ a finite refinement $\calq_e$ of\, $\calp_{\!e}$, such that for every
$e\in\calh$, writing $f_e=f^e_{\mathrm{str}}+ f^e_{\mathrm{err}}+f^e_{\mathrm{unf}}$ where $f^e_{\mathrm{str}}$, $f^e_{\mathrm{err}}$
and $f^e_{\mathrm{unf}}$ are as in \eqref{e2.2}, we have the estimates
\begin{equation} \label{e8.4}
\|f^e_{\mathrm{str}}\|_{L_p} \mik C+1, \ \ \|f^e_{\mathrm{err}}\|_{L_{p^{\dagger}}}\mik \alpha/2 \ \text{ and } \
\|f^e_{\mathrm{unf}}\|_{\cals_{\partial e}}\mik \frac{1}{F(M)}
\end{equation}
where $p^{\dagger}=\min\{2,p\}$. For every $e\in\calh$\, let
\begin{equation} \label{e8.5}
A_e=[f_{\mathrm{str}}^e\mik \zeta], \ \ g_e= f_e\cdot\mathbf{1}_{A_e} \ \text{ and } \ h_e=f_{\mathrm{str}}^e\cdot\mathbf{1}_{A_e}
\end{equation}
and notice that $0\mik g_e\mik \nu_e$ and $0\mik h_e\mik\zeta$. Moreover, by Proposition \ref{p7.2}, we see that
$\|g_e-h_e\|_{\cals_{\partial e}}\mik \alpha$.
\begin{claim} \label{c8.1}
We have $\int \prod_{e\in\calh} h_e\, d\bmu \mik \Delta(n,r,\frac{\ee}{6\zeta})^{n^r}$.
\end{claim}
\begin{proof}
First observe that, by the choice of $\alpha$ and Theorem \ref{t7.1},
\begin{equation} \label{e8.6}
\big|\int \prod_{e\in\calh} g_e\, d\bmu -\int \prod_{e\in\calh} h_e\, d\bmu \big|\mik \delta.
\end{equation}
On the other hand, we have $0\mik g_e\mik f_e$ for every $e\in\calh$. Hence, by \eqref{e8.3} and \eqref{e8.6},
\[ \int\prod_{e\in\calh}h_e\,d\bmu \mik \int\prod_{e\in\calh}f_e\,d\bmu +
\big|\int\prod_{e\in\calh}h_e\, d\bmu - \int \prod_{e\in\calh} g_e\, d\bmu \big| \mik 2\delta. \]
Finally, by \eqref{e8.1}, we have $2\delta\mik \Delta(n,r,\frac{\ee}{6\zeta})^{n^r}$ and the proof is completed.
\end{proof}
Now for every $e\in\calh$ set $E_e=[h_e\meg \Delta(n,r,\frac{\ee}{6\zeta})]$. Since $|\calh|\mik {n\choose r}\mik n^r-1$
and $\Delta(n,r,\frac{\ee}{6\zeta})\mik 1$, by Claim \ref{c8.1} and Markov's inequality, we have
\[ \bmu\Big(\bigcap_{e\in\calh}E_e\Big) \mik \bmu\Big(\Big\{ \bx\in\bbx: \prod_{e\in\calh} h_e(\bx)
\meg \Delta\Big(n,r,\frac{\ee}{6\zeta}\Big)^{|\calh|}\Big\}\Big) \mik \Delta\Big(n,r,\frac{\ee}{6\zeta}\Big). \]
Thus, by Theorem \ref{t3.4}, for every $e\in \calh$ there exists $F_e\in\calb_e$ with
\begin{equation} \label{e8.7}
\bmu(E_e\setminus F_e) \mik \frac{\ee}{6\zeta} \ \text{ and } \ \bigcap_{e\in \calh} F_e=\emptyset.
\end{equation}
Moreover, by \eqref{e8.1}, there exists a collection $\langle\calp_{\! e'}: e'\subseteq e\text{ for some } e\in\calh \rangle$
of partitions of $\bbx$ such that: (i) $\calp_{\!e'}\subseteq \calb_{e'}$ and $|\calp_{\! e'}|\mik k$ for every $e'\subseteq e\in\calh$,
and (ii) for every $e\in\calh$ the set $F_e$ belongs to the algebra generated by the family $\bigcup_{e'\varsubsetneq e} \calp_{\! e'}$.
Therefore, the proof of the theorem will be completed once we show that
\begin{equation} \label{e8.8}
\int_{\bbx\setminus F_e} f_e \, d\bmu\mik \ee
\end{equation}
for every $e\in\calh$. To this end, fix $e\in\calh$ and notice that
\begin{equation} \label{e8.9}
\int_{\bbx\setminus F_e} f_e\, d\bmu \mik \int_{\bbx\setminus F_e} h_e\, d\bmu +
\big|\int_{\bbx\setminus F_e} (g_e-h_e)\, d\bmu\big|+\big|\int_{\bbx\setminus F_e} (f_e-g_e)\,d\bmu\big|.
\end{equation}
Next observe that, by the definition of $E_e$ and the fact that $0\mik h_e\mik \zeta$, we have
\begin{eqnarray} \label{e8.10}
\int_{\bbx\setminus F_e} h_e\, d\bmu & \mik & \int_{\bbx\setminus E_e} h_e\, d\bmu+\int_{E_e\setminus F_e} h_e\, d\bmu \\
& \mik & \Delta\Big(n,r,\frac{\ee}{6\zeta}\Big)+ \zeta\, \bmu(E_e\setminus F_e) \stackrel{\eqref{e8.7}}{\mik} \ee/3. \nonumber
\end{eqnarray}
To estimate the second term in the right-hand side of \eqref{e8.9}, let $\cala$ denote the algebra on $\bbx$ generated by the family
$\bigcup_{e'\varsubsetneq e} \calp_{\!e'}$ and note that every atom of $\cala$ is of the form $\bigcap_{e'\varsubsetneq e} A_{e'}$
where $A_{e'}\in \calp_{\! e'}$ for every $e'\varsubsetneq e$. It follows that the number of atoms of $\cala$ is
less than $k^{2^r}$ and, moreover, every atom of $\cala$ belongs to $\cals_{\partial e}$. In particular, there exists
a family $\calf\subseteq \cals_{\partial e}$ consisting of pairwise disjoint sets with $|\calf|\mik k^{2^r}$ and such that
$\bbx\setminus F_e=\bigcup \calf$. Therefore, by the fact that $\|g_e-h_e\|_{\cals_{\partial e}}\mik \alpha$
and the choice of $\alpha$, we have
\begin{equation} \label{e8.11}
\big| \int_{\bbx\setminus F_e} (g_e-h_e)\, d\bmu\big| \mik \sum_{A\in\calf} \big|\int_{A} (g_e-h_e)\,d\bmu\big|
\mik |\calf|\,\alpha \mik k^{2^{r}}\alpha \mik \varepsilon/3.
\end{equation}
Finally, to estimate the last term in the right-hand side of \eqref{e8.9}, notice that if $p=\infty$, then this term is equal to zero.
(Indeed, in this case we have $\zeta=C+1$ and $A_e=\bbx$.) On the other hand, if $1<p<\infty$, then, by Proposition \ref{p7.2} and
the choice of $\zeta$ and $\alpha$, we obtain that
\begin{eqnarray} \label{e8.12}
\big| \int_{\bbx\setminus F_e} (f_e-g_e)\,d\bmu \big| & = & \int_{\bbx\setminus F_e}f_e\cdot\mathbf{1}_{\bbx\setminus A_e}\,d\bmu
\mik \int_{\bbx\setminus A_e}\!\! f_e \,d\bmu \\
& \mik & (C+1)^p\zeta^{1-p}+\alpha \mik \ee/3. \nonumber
\end{eqnarray}

Combining \eqref{e8.9}--\eqref{e8.12} we conclude that \eqref{e8.8} is satisfied, and so the entire proof of Theorem \ref{t2.2} is completed.
\medskip

\noindent 8.3. \textbf{Consequences.} As we have noted, Theorem \ref{t2.2} yields a Szemer\'{e}di-type result for sparse pseudorandom
subsets of finite additive groups. (Recall that an \emph{additive group} is just an abelian group written additively.) The argument for
deducing this result is well-known---see, e.g., \cite{Go1,RSTT,So,Tao2}---and originates from the work of Ruzsa and Szemer\'{e}di \cite{RS}.

First we introduce some notation. If $A_1,\dots,A_d$ are nonempty finite sets and $f\colon A_1\times\cdots\times A_d\to\rr$
is a function, then by $\ave[f(x_1,\dots,x_d)\, |\, x_1\in A_1,\dots,x_d\in A_d]$ we shall denote the average of $f$, that is,
\[ \ave[f(x_1,\dots,x_d)\, |\, x_1\in A_1,\dots,x_d\in A_d]\coloneqq \frac{1}{|\prod_{i=1}^d A_i|}
\sum_{(x_1,\dots,x_d)\in \prod_{i=1}^d A_i}\!\!\!\!\!\!\!\!\!\!\!\!\! f(x_1,\dots,x_d). \]
The following relative multidimensional Szemer\'{e}di theorem is the main result in this subsection. It follows from
Theorem \ref{t2.2} arguing precisely as in the proof of \cite[Theorem 2.18]{Tao2}.
\begin{thm} \label{t8.2}
For every integer $k\meg 3$, every $C\meg 1$, every $1<p\mik\infty$ and every $0<\delta\mik 1$ there exist a positive integer
$N=N(k,C,p,\delta)$ and a strictly positive constant $c=c(k,C,p,\delta)$ with the following property. Let $Z,Z'$ be finite
additive groups and let $\langle \phi_i : i\in [k]\rangle$ be a collection of group homomorphisms from $Z$ into $Z'$ such that
the set $\{\phi_i(d)-\phi_j(d): i,j\in [k] \text{ and } d\in Z\}$ generates $Z'$. Consider the $(k-1)\text{-uniform}$ hypergraph
system $\mathscr{H}=(k,\langle(X_i,\mu_i):i\in [k]\rangle,\calh)$ where: $\emph{(a)} \ \calh={k\choose k-1}$,
and \emph{(b)} $(X_i,\mu_i)$ is the discrete probability space with $X_i=Z$ and $\mu_i$ the uniform probability measure
on $Z$ for every $i\in [k]$. Also let  $\nu\colon Z'\to \rr$ be a nonnegative function and for every $j\in [k]$ define
$\nu_{[k]\setminus\{j\}}\colon\bbx\to \rr$ by the rule
\begin{equation} \label{e8.13}
\nu_{[k]\setminus\{j\}}\big((x_i)_{i\in [k]}\big)=\nu\Big(\sum_{i\in [k]} \big(\phi_i(x_i)-\phi_j(x_i)\big) \Big).
\end{equation}
$($Here, we have $\bbx=X_1\times\cdots\times X_k)$. Assume that the family $\langle \nu_{[k]\setminus\{j\}}: j\in [k]\rangle$ is
$(C,N^{-1},p)$-pseudorandom and that $|Z|\meg N$. Then for every $f\colon Z'\to \rr$ with $0\mik f\mik \nu$ and\,
$\ave[f(x)\, |\, x\in Z']\meg\delta$ we have
\begin{equation} \label{e8.14}
\ave\Big[ \prod_{j\in [k]} f\big(a+ \phi_j(d)\big)\,\Big|\ a\in Z',d\in Z\Big]\meg c.
\end{equation}
\end{thm}
We close this subsection by discussing in more detail the pseudorandomness hypotheses appearing in Theorem \ref{t8.2} for the important
special case of the cyclic group $\mathbb{Z}_n\coloneqq \mathbb{Z}/n\mathbb{Z}$. Specifically,
let $k,C$ and $p$ be as in Theorem \ref{t8.2} and let $0<\eta\mik 1$. Also let $n$ be a positive integer. We say that a function
$\nu\colon\mathbb{Z}_n\to \rr$ is \emph{$(k,C,\eta,p)$-pseudorandom} if it is nonnegative and satisfies the following conditions.
\begin{enumerate}
\item[(C1)] If $\ell$ is as in \eqref{e6.4}, then we have
\begin{multline*}
1-\eta\mik\ave\Big[ \prod_{j=1}^k \prod_{\omega \in\{0,\dots,\ell-1\}^{[k]\setminus \{j\}}}
\nu\big(\sum_{i=1}^k(i-j)\,x_i^{(\omega_i)}\big)^{n_{j,\omega}} \\
\Big| \, x_1^{(0)},\dots, x_1^{(\ell-1)}, \dots, x_k^{(0)},\dots, x_k^{(\ell-1)}\in \mathbb{Z}_n\Big] \mik C+\eta
\end{multline*}
for any choice of\, $n_{j,\omega}\in\{0,1\}$.
\item[(C2)] For every $j\in [k]$ there exists $\boldsymbol{\psi}_j\colon\mathbb{Z}_n^{[k]\setminus\{j\}}\to \rr$
with $\ave\big[|\bpsi_j|^p\big]\mik C^p$ and
\begin{multline*}
\!\!\!\!\Big|\ave\Big[ \prod_{\omega\in\{0,1\}^{[k]\setminus \{j\}}} \Big(\nu\big(\sum_{i=1}^k (i-j)\,x_i^{(\omega_i)}\big)-
\boldsymbol{\psi}_j\big((x_i^{(\omega_i)})_{i\in [k]\setminus\{j\}}\big)\Big) \, \times \\
\ \ \ \ \times\!\! \prod_{j'\in [k]\setminus\{j\}}\prod_{\omega\in\{0,1\}^{[k]\setminus \{j'\}}}\!\!\! \nu\big(\sum_{i=1}^k (i-j')\,x_i^{(\omega_i)}
\big)^{n_{j'\!,\omega}} \,\Big|\, x_1^{(0)}, x_1^{(1)}, \dots, x_k^{(0)}, x_k^{(1)} \in \mathbb{Z}_n\Big]\Big| \mik \eta
\end{multline*}
for any choice of\, $n_{j'\!,\omega}\in\{0,1\}$.
\end{enumerate}
By Proposition \ref{p6.4} and Theorem \ref{t8.2}, we have the following theorem.
\begin{thm} \label{t8.3}
Let $k\in\nn$ with $k\meg 3$, $C\meg 1$ and $1<p\mik\infty$, and let $0<\delta\mik 1$. Also let $N=N(k,C,p,\delta)$ and
$c=c(k,C,p,\delta)$ be as in Theorem \emph{\ref{t8.2}} and set
\[ \eta=\big((C+1)N\big)^{-2^{k-1}}. \]
If\, $n\in\nn$ with $n\meg N$ and $\nu\colon\mathbb{Z}_n\to \rr$ is $(k,C,\eta,p)$-pseudorandom, then for every $f\colon\mathbb{Z}_n\to \rr$
with $0\mik f\mik \nu$ and\, $\ave[f(x)\,|\,x\in\mathbb{Z}_n]\meg \delta$ we have
\begin{equation} \label{e8.15}
\ave\Big[\prod_{j\in [k]}f\big(a+ j d\big)\,\Big|\ a, d\in\mathbb{Z}_n\Big]\meg c.
\end{equation}
\end{thm}


\end{document}